\documentclass[11pt,a4paper]{article}
\usepackage{amsmath,amsfonts,amssymb,mathrsfs,comment,soul}
\usepackage{graphicx}
\usepackage{xfrac}
\usepackage{subfigure}
\usepackage{xcolor}
\usepackage[english]{babel}
\usepackage{mathscinet}
\usepackage{tikz}
\usetikzlibrary{decorations.pathmorphing}
\usepackage{pgfplots}
\pgfplotsset{compat=1.16}

\newtheorem{theorem}{\bf Theorem}[section]
\newtheorem{lemma}[theorem]{\bf Lemma}
\newtheorem{corollary}[theorem]{\bf Corollary}
\newtheorem{definition}[theorem]{\bf Definition}

\newtheorem{proposition}[theorem]{\bf Proposition}

\newenvironment{proof}{\noindent {\sc Proof.}}{\hfill $\square$}

\newcommand{\R}{\mathbb{R}}

\newcommand{\N}{\mathbb{N}}

\def \M {{\mathcal{M}}} 
\def \N {{\mathcal{N}}} 

\def \L {\mathscr{L}}
\def \H {\mathscr{H}}
\def \Q {\mathcal{Q}}

\def \g {{\gamma}}

\def \r {{\varrho}}

\def \t {{\tau}}

\def \x {{\xi}}

\def \z {{\zeta}}

\def \phi {{\varphi}}

\def \G {{\Gamma}}
\def \O {{\Omega}}

\def \div {{\text{\rm div}}}
\def \loc {{\text{\rm loc}}}

\def\p{\partial}

\def \AS {{ \mathscr{A}_{z_0} }}

\newcommand{\pr}[1]{\left(#1\right)}
\newcommand{\scp}[1]{\langle #1 \rangle}
\newcommand{\mres}{\mathbin{\vrule height 1.6ex depth 0pt width 0.13ex\vrule height 0.13ex depth 0pt width 1.3ex}}

\DeclareRobustCommand{\rchi}{{\mathpalette\irchi\relax}}
\newcommand{\irchi}[2]{\raisebox{\depth}{$#1\chi$}}
\newcommand{\norm}[1]{\left\lVert#1\right\rVert}

\usepackage{mathtools}
\DeclarePairedDelimiter{\abs}{\lvert}{\rvert}

\definecolor{bblue}{rgb}{.6 .6 1}
\definecolor{lblue}{rgb}{.8 .8 1}
\definecolor{ddblue}{rgb}{.4 .4 1}
\definecolor{dddblue}{rgb}{.28 .28 .7}
\definecolor{ddddblue}{rgb}{.02 .02 .6}

\begin{document}
\title{Mean value formulas for classical solutions to uniformly parabolic equations in divergence form}
\author{{\sc{Emanuele Malagoli}
\thanks{Dipartimento di Scienze Fisiche, Informatiche e Matematiche, Universit\`{a} degli Studi di Modena e Reggio Emilia, via Campi 213/b, 41125 Modena (Italy). E-mail: 169864@studenti.unimore.it}, \ 
\sc{Diego Pallara}
\thanks{Dipartimento di Matematica e Fisica ``Ennio De Giorgi'', Universit\`{a} del Salento and INFN, Sezione di Lecce, Ex Collegio Fiorini - Via per Arnesano - Lecce (Italy). E-mail: 	diego.pallara@unisalento.it} \ 
\sc{Sergio Polidoro}
\thanks{Dipartimento di Scienze Fisiche, Informatiche e Matematiche, Universit\`{a} degli Studi di Modena e Reggio Emilia, Via Campi 213/b, 41125 Modena (Italy). E-mail: 	sergio.polidoro@unimore.it}
}}
	
\date{ }
	
\maketitle

\bigskip

\begin{abstract}
We prove surface and volume mean value formulas for classical solutions to uniformly parabolic equations in divergence form. 
We then use them to prove the parabolic strong maximum principle and the parabolic Harnack inequality. 
We emphasize that our results  only rely on the classical theory, and our arguments follow the lines used in the original theory of harmonic functions. We provide two proofs relying on two different formulations of the divergence theorem, one stated for sets with {\em almost $C^1$-boundary}, the other stated for sets with finite perimeter.
\end{abstract}
	
\setcounter{equation}{0} 
\section{Introduction}\label{secIntro}
Let $\Omega$ be an open subset of $\R^{N+1}$. We consider classical solutions $u$ to the equation $\L u = f$ in $\Omega$, where $\L$ is a parabolic operator in divergence form defined for $z=(x,t) \in \R^{N+1}$ as follows
\begin{equation} \label{e-L}
	\L u (z)  := \sum_{i,j=1}^{N}\tfrac{\p}{\p x_i} \left( a_{ij} (z) \tfrac{\p u}{\p x_j}(z) \right) + 
	\sum_{i=1}^{N} b_{i} (z) \tfrac{\p u}{\p x_i}(z) + c(z) u(z) - \, \tfrac{\p u}{\p t}(z).
\end{equation}
In the following we use the notation 
$A(z) := \left( a_{ij}(z) \right)_{i,j=1,\dots, N}, b(z) := \left( b_{1} (z), \dots, b_{N} (z) \right)$ and we write $\L u $ in the short form
\begin{equation} \label{e-LL}
	\L u (z)  := \div \left( A (z) \nabla_x u(z) \right) + \langle b(z), \nabla_x u(z)\rangle + c(z) u(z) - 
	\, \tfrac{\p u}{\p t}(z).
\end{equation}
Here $\div, \nabla_x$ and $\langle \, \cdot \, , \, \cdot \, \rangle$ denote the divergence, the gradient and the inner product in $\R^N$, respectively. We assume that the matrix $A(z)$ is symmetric and that the coefficients of the operator $\L$ are H\"older continuous functions with respect to the parabolic distance. This means that there exist two constants $M>0$ and $\alpha \in ]0,1]$, such that
\begin{equation} \label{e-hc}
	|c (x,t) - c (y,s)| \le M \left(|x-y|^\alpha + |t-s|^{\alpha/2}\right),
\end{equation}
for every $(x,t), (y,s) \in \R^{N+1}$. We require that the above condition is satisfied not only by $c$, 
but also by $a_{ij}, \frac{\p a_{ij}}{\p x_i}, b_{i}, \frac{\p b_{i}}{\p x_i}$, for $i, j = 1, \dots, N$, 
with the same constants $M$ and $\alpha$. We finally assume that the coefficients of $\L$ are bounded and 
that $\L$ is uniformly parabolic, \emph{i.e.}, there exist two constants $\lambda, \Lambda$, with $0 < \lambda < \Lambda$, such that
\begin{equation} \label{e-up}
	\lambda |\xi|^2  \le \langle A(z) \xi, \xi \rangle \le \Lambda |\xi|^2, 
	\quad \left|\tfrac{\p a_{ij}}{\p x_i}\right| \le \Lambda, 
	\quad |b_i(z)| \le \Lambda, \quad |c(z)| \le \Lambda,
\end{equation}
for every $\xi \in \R^N$, for every $z \in \R^{N+1}$, and for $i,j=1, \dots, N$. Under the above assumptions, the classical parametrix method provides us with the existence of a fundamental solution $\Gamma$. In Section \ref{secFundSol} we shall quote from the monograph of Friedman \cite{Friedman} the results we need for our purposes. 

The main achievements of this note are some mean value formulas for the solutions to $\L u = f$ that are written in terms of the level and super-level sets of the fundamental solution $\Gamma$. We extend previous results of Fabes and Garofalo \cite{FabesGarofalo} and Garofalo and Lanconelli \cite{GarofaloLanconelli-1989} in that we weaken the regularity requirement on the coefficients of $\L$ that in \cite{FabesGarofalo, GarofaloLanconelli-1989} are assumed to be $C^\infty$ smooth. As applications of the mean value formulas we give an elementary proof of the parabolic strong maximum principle. 
We note that the conditions on the functions $\frac{\p a_{ij}}{\p x_i}$'s are needed in order to deal with classical solutions to the adjoint equation $\L^* v = 0$, as the mean value formulas rely on the divergence theorem applied to the function $(\x,\tau) \mapsto \Gamma(x,t,\xi,\tau)$.

We introduce some notation in order to state our main results. For every $z_0=(x_0, t_0) \in \R^{N+1}$ and 
for every $r>0$, we set
\begin{equation} \label{e-Psi}
\begin{split}
	\psi_r(z_0) & := \left\{ z \in \R^{N+1} \mid \Gamma(z_0; z ) = \tfrac{1}{r^N} \right\}, 
	\\
	\Omega_r(z_0) & := \left\{ z \in \R^{N+1} \mid \Gamma(z_0; z) > \tfrac{1}{r^N} \right\}.
\end{split}
\end{equation}

\begin{center}
\begin{tikzpicture}
\clip (-.51,7.21) rectangle (6.31,1.99);
\path[draw,thick] (-.5,7.2) rectangle (6.3,2);
\begin{axis}[axis x line=middle, axis y line=middle,
    xtick=\empty,ytick=\empty, 
    ymin=-1.2, ymax=1.2, xmin=-.2,xmax=1.7, samples=121, rotate= -90]
\addplot [ddddblue,line width=.7pt,domain=.001:.1111] {sqrt(- 3 * x * ln(9*x)};
\addplot [ddddblue,line width=.7pt,domain=.001:.1111] {- sqrt(- 3 * x * ln(9*x))};
\addplot [dddblue,line width=.7pt,domain=.001:.25] {sqrt(- 2 * x * ln(4*x)};
\addplot [dddblue,line width=.7pt,domain=.001:.25] {- sqrt(- 2 * x * ln(4*x))};
\addplot [ddblue,line width=.7pt, domain=.001:1] {sqrt(- 2 * x * ln(x))}; 
\addplot [ddblue,line width=.7pt,domain=.001:1] {- sqrt(- 2 * x * ln(x))};
\addplot [black,line width=1pt, domain=-.01:.01] {sqrt(.0001 - x * x)} node[above] {\quad $z_0$};
\addplot [black,line width=1pt, domain=-.01:.01] {-sqrt(.0001 - x * x)} 
node[below]{\qquad \qquad \qquad \qquad \quad \quad \quad \quad {\color{ddddblue} $\psi_r(z_0)$}};
\end{axis}
\draw [<-,line width=.4pt] (2.8475,7) -- (2.8475,2);
\end{tikzpicture} 

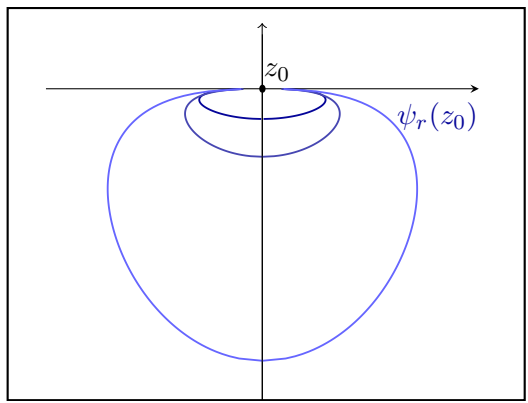
{{\sc Fig.1}} - $\psi_r(z_0)$ for three different values of $r$.
\end{center}

Similarly to the elliptic case, we call $\psi_r(z_0)$ and $\Omega_r(z_0)$ respectively the \emph{parabolic sphere} and the \emph{parabolic ball} with radius $r$ and ``center'' at $(x_0,t_0)$. Note that, unlike the elliptic setting, $z_0$ belongs to the topological boundary of $\Omega_r(z_0)$. Because of the properties of the fundamental solution of uniformly parabolic operators, the parabolic balls $\Omega_r(z_0)$ are bounded sets and shrink to the center $z_0$ as $r \to 0$. We finally introduce the following kernels 
\begin{equation} \label{e-kernels}
\begin{split}
	K (z_0; z) & := 
	\frac{\langle A(z) \nabla_x \Gamma(z_0; z), \nabla_x \Gamma(z_0; z) \rangle }
	{|\nabla_{(x,t)}\Gamma(z_0;z)|},
	\\
	M(z_0; z) & := 
	\frac{\langle A(z) \nabla_x \Gamma(z_0; z), \nabla_x \Gamma(z_0; z) \rangle }
	{\Gamma(z_0; z)^2}.
\end{split}
\end{equation}
Here $\nabla_x \Gamma(z_0; z)$ and $|\nabla_{(x,t)}\Gamma(z_0; z)|$ denote the gradient with respect to the space variable $x$ and the norm of the gradient with respect to the variables $(x,t)$ of $\Gamma$, respectively. Moreover, we agree to set $K (z_0; z) = 0$ whenever $\nabla_{(x,t)}\Gamma(z_0;z)=0$. In the following, $\H^{N}$ denotes the $N$-dimensional Hausdorff measure. The first achievements of this note are the following mean value formulas.

\begin{theorem} \label{th-1}
 Let $\Omega$ be an open subset of $\R^{N+1}$, $f\in C(\Omega)$ and let $u$ be a classical solution to $\L u = f$ in $\Omega$. Then, for every $z_0 \in \Omega$ and for almost every $r>0$ such that $\overline{\Omega_r(z_0)} \subset \Omega$ we have 
 \begin{align*} 
	u(z_0) = \int_{\psi_r(z_0)} K (z_0; z) u(z) \, d \H^{N} (z) 
	+ & \int_{\Omega_r(z_0)} f (z) \left( \tfrac{1}{r^N} - \Gamma(z_0; z) \right)\ dz  \\
	 + & \frac{1}{r^N} \int_{\Omega_r(z_0)} \left( \div \, b(z) - c(z) \right) u(z) \ dz, 
	\\
	u(z_0) = \frac{1}{r^N} \int_{\Omega_r(z_0)} \!\!\!\!\!  M (z_0; z) u(z) \, dz 
	+ \frac{N}{r^N} & \int_0^{r} \left(\r^{N-1}
	\int_{\Omega_\r (z_0)} f (z) \left( \tfrac{1}{\r^N} - \Gamma(z_0;z) \right) dz \right) d \r  \\
    + &  \frac{N}{r^N} \int_0^{r} \left(\frac{1}{\r}
	\int_{\Omega_\r (z_0)} \left( \div \, b(z) - c(z) \right) u(z) \, dz \right) d \r.
\end{align*}
The second statement holds for \emph{every} $r>0$ such that ${\Omega_r(z_0)} \subset \Omega$.
\end{theorem}

Note that $\tfrac{1}{r^N} - \Gamma(z_0; z) < 0$ in the set $\Omega_r(z_0)$, because of its very definiton \eqref{e-Psi}. This fact, together with the non-negativity of the kernels \eqref{e-kernels} will be used in the sequel to obtain the strong maximum principle from Theorem \ref{th-1}.

We next put Theorem \ref{th-1} in its context. It restores the mean value formulas first proved by Pini in \cite{Pini1951} for the heat equation $\p_t u = \p_x^2 u$, then by Watson in \cite{Watson1973} for the heat equation in several space variables. We also recall the mean value formulas first proved by Fabes and Garofalo in \cite{FabesGarofalo} for the equation $\L u = 0$, then extended by Garofalo and Lanconelli \cite{GarofaloLanconelli-1989} to the equation $\L u = f$, where the operator $\L$ has the form \eqref{e-L} and its coefficients are assumed to be $C^\infty$ smooth. This extra regularity assumption on the coefficients of $\L$ is due to the fact that the mean value formula relies on the divergence theorem applied to the parabolic ball $\Omega_r(z_0)$. Since the explicit epression of the fundamental solution $\Gamma$ is not available when the coefficients of $\L$ are variable, the authors of \cite{FabesGarofalo} and \cite{GarofaloLanconelli-1989} rely on the Sard theorem (see \cite{Sard}) which guarantees that $\psi_r(z_0)$ is a manifold for \emph{almost every} positive $r$, provided that the fundamental solution $\Gamma$ is $N+1$ times differentiable. The smoothness of the coefficients of the operator $\L$ is used in \cite{FabesGarofalo} and \cite{GarofaloLanconelli-1989} in order to have the needed regularity on $\Gamma$. 

The main goal of this note is the restoration of natural regularity hypotheses for the existence of classical solutions to $\L u =f$. These assumptions can be further weakened, since the existence of a fundamental solution has been proved for operators with Dini continuous coefficients. We prefer to keep our treatment in the usual setting of H\"older continuous functions for the sake of simplicity. The unnecessary regularity conditions on the coefficients of $\L$ can be removed in two ways. Following an approach close to the classical one, it is possible to rely on a result due to Dubovicki\u{\i} \cite{Dubovickii} (see also Bojarski, Ha{j\l}asz, and Strzelecki \cite{BojHajStrz}) which allows to reduce the regularity requirement on $\Gamma$ in order to apply a generalized divergence theorem for {\em sets with almost $C^1$ boundary}. This is presented in Section \ref{SectionDivergence} and applied in Section \ref{sectionProof}. The other approach relies on geometric measure theory and is presented in the last section: we show how the proof of Theorem \ref{th-1} can be modified relying on the generalized divergence theorem proved by De Giorgi \cite{DeGiorgi2,DeGiorgi3} 
in the framework of finite perimeter sets. As said before, this deep theory is not necessary in the present context, but it is more flexible and its generalization to Carnot groups (where the analogue of Dubovicki\u{\i}'s Theorem is not available) will allow us to extend the results of the present paper to {\em degenerate} parabolic operators. We have presented the application to uniformly parabolic operators to pave the way to this generalization, which will be the subject of a forthcoming paper. 

\medskip

The mean value formulas stated in Theorem \ref{th-1} provide us with a simple proof of the strong maximum (minimum) principle for the operator $\L$ when $c = 0$. Note that, in this case, the constant function $u(x,t) = 1$ is a solution to $\L u = 0$, so that the mean value formula  gives $\frac{1}{r^N}\int_{\Omega_r(z_0)}M(z_0;z)dz=1$. In order to state this result we first introduce the notion of \emph{attainable set}. We say that a curve $ \g: [0,T] \rightarrow \R^{N+1}$ is \emph{$\L$-admissible} if it is absolutely continuous and 
\begin{equation*}
	 \dot{\g}(s) = \left( \dot {x}_1(s), \dots, \dot {x}_N(s), -1 \right)
\end{equation*}	
for almost every $s \in [0,T]$, with $\dot{x}_1, \dots, \dot{x}_N \in L^{2}([0,T])$.

\begin{definition} \label{def-prop-set}
Let $\O$ be any open subset of $\R^{N+1}$, and let $z_0 \in \O$. The \emph{attainable set} is 
\begin{equation*}
	\AS  ( \O ) = 
	\begin{Bmatrix}
	z \in \O \mid \hspace{1mm} \text{\rm there exists an} \ \L - \text{\rm admissible curve}
	\ \g : [0,T] \rightarrow \O \hspace{1mm} \\ 
	\hfill \text{\rm such that} \ \g(0) = z_0 \hspace{1mm} {\rm and}  
	\hspace{1mm} \g(T) = z
	\end{Bmatrix}.
\end{equation*}
Whenever there is no ambiguity on the choice of the set $\O$ we denote $\AS = \AS ( \O )$. 
\end{definition}

\begin{proposition} \label{prop-smp}
Let $\O$ be any open subset of $\R^{N+1}$, and suppose that $c = 0$. Let $z_0=(x_0,t_0) \in \O$ and let $u$ 
be a classical solution to 
$\L u = f$. If $u (z_0) = \max_\Omega u$ and $f \ge 0$ in $\Omega$, then 
\begin{equation*}
u(z) = u(z_0) \quad \text{and} \quad f(z) = 0 \qquad \text{for every} \ z \in \overline {\AS  ( \O )}.
\end{equation*}
The analogous result holds true if $u (z_0) = \min_\Omega u$ and $f \le 0$ in $\Omega$. 
\end{proposition}

If we remove the assumption $c =0$ we obtain the following weaker result. 

\begin{proposition} \label{prop-smp1}
Let $\O$ be any open subset of $\R^{N+1}$. Let $u \le 0$ ($u \ge 0$, respectively) be a classical solution to 
$\L u = f$ with $f \ge 0$ ($f \le 0$, respectively) in $\Omega$. If $u (z_0) = 0$ for some $z_0 \in \Omega$, then 
\begin{equation*}
u(z) = 0 \quad \text{and} \quad f(z) = 0 \qquad \text{for every} \ z \in \overline {\AS  ( \O )}.
\end{equation*}
\end{proposition}

In the remaning part of this intoduction we focus on some modified mean value formulas useful in the proof of parabolic Harnack inequality. As already noticed, the main difficulty one encounters in the proof of the Harnack inequality is due to the unboundedness of the kernels introduced in \eqref{e-kernels}. In order to overcome this issue, we can rely on the idea introduced by Kupcov in \cite{Kupcov4}, and developed by Garofalo and Lanconelli in \cite{GarofaloLanconelli-1989} in the case of parabolic operators with smooth coefficients. This method provides us with some bounded kernels and gives us a useful tool for a direct proof of the Harnack inequality. We outline here the procedure. Let $m$ be a positive integer, and let $u$ be a solution to $\L u = f$ in $\R^{N+1}$. We set

\begin{equation*}
 \widetilde u(x,y,t) := u(x,t), \qquad \widetilde f(x,y,t) := f(x,t),  \qquad (x,y,t) \in \R^{N}\times \R^{m} \times \R,
\end{equation*}
and we note that 
\begin{equation*}
 \widetilde \L \ \widetilde u(x,y,t) = \widetilde f(x,y,t) \qquad \widetilde \L 
 = \L + \sum_{j=1}^{m}\tfrac{\partial^2}{\partial y^2_j} = \L + \Delta_y.
\end{equation*}
Moreover, if $\Gamma$ and $K_m$ denote fundamental solutions of $\L$ and of the heat equation in $\R^m$, respectively, then the function
\begin{equation*}
 \widetilde \Gamma(\xi, \eta, \tau ;x,y,t) = \Gamma (\xi, \tau ;x,t) K_m (\eta,\tau ;y,t)
\end{equation*}
is a fundamental solution of $\widetilde \L$. Then, integrating with respect to $y$ in the mean value formulas of Theorem \ref{th-1}, applied to $\widetilde u$ and to the operator $\widetilde \L$, gives new kernels, that are bounded whenever $m>2$. We intoduce further notations.
\begin{equation} \label{e-Omegam}
\begin{split}
	\Omega^{(m)}_r(z_0) := & \left\{ z \in \R^{N+1} \mid (4 \pi (t_0-t))^{-m/2}\Gamma(z_0; z) > 
	\tfrac{1}{r^{N+m}} \right\}, \\
	N_r(z_0;z) := & 2 \sqrt{t_0-t}\sqrt{\log\left(\tfrac{r^{N+m}}{ (4 \pi (t_0-t))^{m/2}}  \Gamma(z_0;z) \right)},
	\\
	M_r^{(m)} (z_0;z):= & \omega_m N_r^m(z_0;z) \left( M(z_0;z) + \frac{m}{m+2} \cdot \frac {N_r^2(z_0;z)}{4(t_0-t)^2} \right),\\
	W_r^{(m)} (z_0;z):= & \frac{\omega_m}{r^{N+m}} N_r^m(z_0;z) - \frac{m}{2} \cdot \frac{\omega_m}{(4 \pi)^{m/2}}  \, 
	\Gamma(z_0,z) \cdot \widetilde \gamma \left(\frac{m}{2}; \frac {N_r^2(z_0;z)}{4(t_0-t)} \right),
\end{split}
\end{equation}
where $M(z_0;z)$ is the kernel introduced in \eqref{e-kernels}, $\omega_m$ denotes the volume of the $m$-dimensional unit ball 
and $\widetilde \gamma$ is the lower incomplete gamma function
\begin{equation*}
  \widetilde \gamma(s;w) := \int_0^w \tau^{s-1} e^{- \tau} d \tau.
\end{equation*}
Note that the function $N(z_0,z)$ is well defined for every $z \in \Omega^{(m)}_r(z_0)$, as the argument of 
the logarithm is positive, and that we did not point out the dependence of $N_r$ on the space dimension
$m$ to avoid a possible confusion with its powers appearing in the definitions of $M_r^{(m)}$ and 
$W_r^{(m)}$.
\begin{proposition} \label{prop-2}
 Let $\Omega$ be an open subset of $\R^{N+1}$, and let $u$ be a classical solution to $\L u = f$ in $\Omega$. Then, for every $z_0 \in \Omega$ and for every $r>0$ such that ${\Omega^{(m)}_r(z_0)} \subset \Omega$ we have 
 \begin{equation*} 
 \begin{split}
    u(z_0) = & \frac{1}{r^{N+m}} \int_{\Omega^{(m)}_r(z_0)} \! \! \! \! \! \! M_r^{(m)} (z_0; z) u(z) \, dz \,
 	\\
 	& + \frac{N+m}{r^{N+m}} \int_0^{r} \left( \r^{N+m-1} \int_{\Omega_\r^{(m)} (z_0)} 
 	\! \! \! \! W_\r^{(m)} (z_0;z) f (z) \, dz \right) d \r \\
 	& + \frac{N+m}{r^{N+m}} \int_0^{r} \left( \frac{\omega_m}{\r} \int_{\Omega_\r^{(m)} (z_0)} \! \! \! \! N_\r^m(z_0;z)
 	 \left( \div \, b(z) - c(z) \right) u(z) \, dz \right) d \r.
 \end{split}
\end{equation*}
\end{proposition}

\medskip

We conclude this introduction with two statements of the parabolic Harnack inequality. The first one is given in terms of the parabolic ball $\Omega_{r}^{(m)}(z_0)$, the second one is the usual invariant parabolic Harnack inequality. We emphasize that our proof is elementary, as it is based on the mean value formula, however some accurate estimates of the fundamental solution are needed in order to control the Harnack constant and the size of the cylinders appearing in its statement. 
For every $z_0 = (x_0,t_0) \in \R^{N+1}, r >0$, and $m \in \mathbb{N}$ we set
\begin{equation} \label{e-Kr}
	K^{(m)}_{r}(z_0) := \overline{\Omega_{r}^{(m)}(z_0)} \cap \Big\{ t \le t_0 - \frac{1}{4 \pi \lambda^{N/(N+m)}} \,r^2 \Big\}.
\end{equation}
We note that, as a consequence of Lemma \ref{lem-localestimate*} below, for every sufficiently small $r$ the compact set $K^{(m)}_{r}(z_0)$ is non empty.

\begin{proposition} \label{prop-Harnack}
For every $m \in \mathbb{N}$ with $m > 2$, there exist two positive constants $r_0$ and $C_K$, only depending on $\L$ and $m$, such that the following inequality holds. Let $\Omega$ be an open subset of $\R^{N+1}$. For every $z_0 \in \Omega$ and for every positive $r$ such that $r \le r_0$ and $\Omega_{5r}^{(m)}(z_0) \subset \Omega$ we have that
  \begin{equation} \label{e-H0}
	\sup_{K^{(m)}_{r}(z_0)} u \le C_K u(z_0)
\end{equation}
for every $u \ge 0$ solution to $\L u = 0$ in $\Omega$.
\end{proposition}

We introduce some further notation in order to state an \emph{invariant} Harnack inequality. For every $z_0 = (x_0,t_0) \in \R^{N+1}$ and for every $r>0$ we set
\begin{equation} \label{e-Q}
	\Q_{r}(z_0) := B_r(x_0) \times ]t_0- r^2,t_0[,
\end{equation}
where $B_r(x_0)$ denotes the Euclidean ball with center at $x_0$ and radius $r$. Moreover, for $0 < \iota < \kappa < \mu < 1$ and $0 < \vartheta < 1$ we set
\begin{equation} \label{e-QPM}
	\Q^-_{r}(z_0) := B_{\vartheta r}(x_0) \times ]t_0 - \kappa r^2,t_0 - \mu r^2[, \qquad
	\Q^+_{r}(z_0) := B_{\vartheta r}(x_0) \times ]t_0 - \iota r^2,t_0[.
\end{equation}
We have

\begin{theorem} \label{th-Harnack-inv}
Choose positive constants $R_0$ and $\iota, \kappa, \mu, \vartheta$ as above and let $\Omega$ be an open subset of $\R^{N+1}$. Then there exists a positive constant $C_H$, only depending on $\L$, on $R_0$ and on the constant that define the cylinders $\Q, \Q^+, \Q^-$, such that the following inequality holds. For every $z_0 \in \Omega$ and for every positive $r$ such that $r \le R_0$ and ${\Q_{r}(z_0)} \subset \Omega$ we have that
  \begin{equation} \label{e-H1}
	\sup_{\Q^-_{r}(z_0)} u \le C_H \inf_{\Q^+_{r}(z_0)} u
\end{equation}
for every $u \ge 0$ solution to $\L u = 0$ in $\Omega$.
\end{theorem}

\begin{center}
\begin{tikzpicture}
\path[draw,thick] (-1,.7) rectangle (9,6.8);
\filldraw [fill=blue!20!white, draw=blue, line width=.6pt] (1.5,6) rectangle node {{\color{blue} $Q^+_r(x_0,t_0)$}} node [below=.7cm,right=1.95cm] {{\color{blue} $t_0 - \iota r^2$}} (5.5,4.5);
\filldraw [fill=red!20!white, draw=red, line width=.6pt] (1.5,2) rectangle node {{\color{red} $Q^-_r(x_0,t_0)$}} node [above=.7cm,right=1.95cm] {{\color{red} $t_0 - \kappa r^2$}} node [below=.7cm,right=1.95cm] {{\color{red} $t_0 - \mu r^2$}} (5.5,3.5);
\draw [line width=.6pt] (0,6) rectangle node [above=2cm,right=3.6cm] {$Q_r(x_0,t_0)$} node [below=2.2cm,right=3.6cm] {$t_0-r^2$} (7,1);
\draw [line width=.6pt] (3.5,6) circle (.6pt) node[above] {$(x_0,t_0)$} node[above= 7pt, right=1.2cm] {{\color{blue}$\vartheta r$}} node[above=7pt, right=2.8cm] {$r$};
\end{tikzpicture} 

{\sc \qquad Fig.2}  - The set $Q_r(x_0,t_0)$.
\end{center}
 
We conclude this introduction with some comments about our main results. Mean value formulas don't require the uniqueness of the fundamental solution $\Gamma$. In Section \ref{secFundSol} we recall the main results we need on the existence of a fundamental solution together with some known facts about its uniqueness. We also recall in Proposition \ref{prop-localestimate} an asymptotic bound of $\Gamma$ which allows us to use a direct procedure in a part of the proof of the mean value formulas stated in Theorem \ref{th-1}. We point out that recent progresses on mean value formulas and their applications can be found e.g. in \cite{CMPCS}. Moreover, an alternative and more general approach has been introduced by Cupini and Lanconelli in \cite{CupiniLanconelli2021}, where a wide family of differential operators with smooth coefficients is considered. We continue the outline of this article. Section \ref{SectionDivergence} contains the statement of a generalized divergence theorem for {sets with almost $C^1$ boundary} that is used in Section \ref{sectionProof} for the proof of the mean values formulas. Section \ref{sectionHarnack} is devoted to the proof of the Harnack inequality. In the last section we present an alternative approach for the mean value formula that relies on geometric measure theory.

We finally remark that our method also applies to uniformly elliptic equations. Moreover, mean value formulas and Harnack inequality are fundamental tools in the development of the Potential Theory for the operator $\L$. 

\setcounter{equation}{0} 
\section{Fundamental solution}\label{secFundSol}

In this Section we recall some notations and some known results on the classical theory of uniformly parabolic equations that will be used in the sequel. Points of $\R^{N+1}$ are denoted by $z=(x,t), \z = (\x,\t)$ and $\Omega$ denotes an open subset of $\R^{N+1}$. 

Let $u$ be a real valued function defined on $\Omega$. We say that $u$ belongs to $C^{2,1}(\O)$ if 
$u, \frac{\p u}{\p x_j}, \frac{\p^2 u}{\p x_i\p x_j}$ for $i,j= 1, \dots, N$ and $\frac{\p u}{\p t}$ 
are continuous functions, it belongs to $C^{2+ \alpha,1 + \alpha/2}(\O)$ if $u$ and  all the derivatives 
of $u$ listed above belong to the space $C^{\alpha}(\O)$ of the H\"older continuous functions defined by 
\eqref{e-hc}. A function $u$ belongs to $C^\alpha_\loc(\Omega)$ ($C^{2+ \alpha,1 + \alpha/2}_\loc(\O)$,
respectively) if it belongs to $C^\alpha(K)$ (resp. $C^{2+ \alpha,1 + \alpha/2}(K)$) for every compact set 
$K \subset \O$. Let $f$ be a continuous function defined on $\O$. We say that $u \in C^{2,1}(\O)$ is 
a classical solution to $\L u = f$ in $\Omega$ if the equation \eqref{e-L} is satisfied at every point 
$z \in \O$. 

According to Friedman \cite{Friedman}, we say that a fundamental solution $\G$ for the operator $\L$ is a function $\G = \G(z; \z)$ defined for every $(z; \z) \in \R^{N+1} \times \R^{N+1}$ with $t> \tau$, which satisfies the following contitions:
\begin{enumerate}
	 	\item For every $\z = (\xi, \t) \in \R^{N+1}$ the function $\G( \, \cdot \, ; \z)$ belongs to 
	 	$C^{2,1}(\R^{N} \times ]\t, + \infty[)$ and is a classical solution to 
	 	$\L \, \G(\cdot \,; \z) = 0$ in $\R^{N} \times ]\t, + \infty[$;
		\item for every $\phi \in C_c(\R^N)$ the function
$$
  u(z)=\int_{\R^{N}}\Gamma(z;\xi, \t)\phi(\x)d \x, 
$$
is a classical solution to the Cauchy problem
\begin{equation*} 
\left\{
  \begin{array}{ll}
    \L u = 0, & \hbox{$z \in \R^{N} \times ]\t, + \infty[$} \\
    u( \cdot,\t) = \varphi & \hbox{in} \ \R^N.
  \end{array}
\right.
\end{equation*}
\end{enumerate}
Note that $u$ is defined for $t>\t$, then the above identity is understood as follows: for every $\x \in \R^N$ we have $\lim_{(x,t) \to (\x,\t)}u(x,t) = \varphi (\x)$. We also point out that the two above conditions do not guarantee the uniqueness of the fundamental solution. However, as we shall see in the following, estimates \eqref{upper-bound} and \eqref{e-deriv-bds} hold for the fundamental solution $\Gamma$ built by the parametrix method and the fundamental solution verifying such estimates is unique. Indeed, it follows from the proof of Theorem 15 in Ch.1 of \cite{Friedman} that there is only one fundamental solution under the further assumptions that $\Gamma(x,t;\x,\t) \to 0$ as $|x| \to + \infty$ and $|\partial_{x_j}\Gamma(x,t;\x,\t)| \to 0$ as $|x| \to + \infty$, for $j=1, \dots,N$, uniformly with respect to $t$ varying in bounded intervals of the form $]\tau, \tau + T]$.

We outline here the parametrix method for the construction of a fundamental solution $\G$ of $\L$. 
We first note that, if the matrix $A$ in the operator $\L$ is constant, then the fundamental solution of 
$\L$ is explicitly known
\begin{equation} \label{eq-fundsol-A}
 \G_A(z;\z) = \frac{1}{\sqrt{(4 \pi (t-\t))^{N}\det A}} 
 \exp \left( - \frac{\langle A^{-1}(x-\x),x-\x \rangle}{4(t-\t)} \right),
\end{equation}
and moreover the \emph{reproduction property} holds:
\begin{equation} \label{eq-rep-prop-A}
 \G_A(z;\z) = \int_{\R^N} \G_A(x,t;y,s) \G_A(y,s;\x,\t) d y,
\end{equation}
for every $z=(x,t), \z=(\x,\t) \in \R^{N+1}$ and $s \in \R$ with $\t < s < t$. 
A direct computation shows that, for every $T>0$, and $\Lambda^+ > \Lambda$ as in \eqref{e-up}, 
there exists a positive constant $C^+= C^+(\lambda, \Lambda, \Lambda^+,T)$ such that 
\begin{equation} \label{eq-fundsol-bd-2}
\begin{split}
 \left| \frac{\partial \G_A}{\partial {x_j}}(z;\z) \right|  \le \frac{C^+}{\sqrt{t-\t}} \G^+(z;\z), 
 \qquad
 \left| \frac{\partial^2 \G_A}{\partial{x_i x_j}}(z;\z) \right|  \le \frac{C^+}{t-\t} \G^+(z;\z)
\end{split}
\end{equation}
for any $i,j= 1, \dots, N$ and for every $z, \z \in \R^{N+1}$ such that $ 0 < t - \t \le T$. 
Here the function
\begin{equation} \label{eq-fundsol+}
 \G^+(z;\z) = \frac{1}{(\Lambda^+4\pi (t-\t))^{N/2}} \exp \left( - \frac{|x- \x|^2}{4 \Lambda^+(t-\t)} \right),
\end{equation}
is the fundamental solution of $\Lambda^+ \Delta - \frac{\partial}{\partial t}$.
The \emph{parametrix} $Z$ for $\L$ is defined as
\begin{equation} \label{eq-fundsol-Z}
 Z(z;\z) := \G_{A(\z)}(z;\z) = 
 \frac{1}{\sqrt{(4 \pi (t-\t))^{N}\det A(\z)}} \exp \left( - 
 \frac{\langle A(\z)^{-1}(x-\x),x-\x \rangle}{4(t-\t)} \right).
\end{equation}
More specifically, for every fixed $(\x,\tau) \in \R^{N+1}, Z( \, \cdot \, ; \x,\t)$ is the fundamental solution of the operator $\L_{\z}$ obtained by \emph{freezing} the coefficients $a_{ij}$'s of the operator $\L$ at the point $\z$:
\begin{equation} \label{e-L0}
	\L_{\z}  := \div \left( A (\z) \nabla_x \right) - \, \tfrac{\p }{\p t}.
\end{equation}
Note that 
\begin{equation} \label{eq-LZ}
 \L Z(z;\z) := \div \left[\left( A(z) - A(\z) \right) \nabla_x Z(z;\z) \right],
\end{equation}
which vanishes as $z \to \z$, by the continuity of the matrix $A$. The fundamental solution $\G$ for $\L$ is obtained from $Z$ by an iterative procedure. We define the sequence of functions 
$(\L Z)_1 (z;\z) := \L Z(z;\z)$,
\begin{equation} \label{eq-LkZ}
 (\L Z)_{k+1}(z;\z) := \int_\t^t \bigg( \int_{\R^N}(\L Z)_{k} (x,t;y,s) \L Z(y,s;\x,\t) d y \bigg) d s, 
 \qquad k \in \mathbb{N}.
\end{equation}
Note that estimates \eqref{eq-fundsol-bd-2} also apply to $Z$ then, by using the H\"older continuity of the coefficients of $\L$, we obtain
\begin{equation*} 
 \left|\L Z(z;\z) \right| \le \frac{\widetilde C}{(t-\t)^{1 - \alpha/2}} \G^+(z;\z),
\end{equation*}
for a positive constant $\widetilde C$ depending on $\lambda, \Lambda, \Lambda^+, T$ and on the constant $M$ in \eqref{e-hc}. 
This inequality and the reproduction property \eqref{eq-rep-prop-A} applied to $\Gamma^+$ imply that, for every $k \ge 2$, the integral that defines $(\L Z)_{k}$ converges  and 
\begin{equation*} 
 \left| (\L Z)_{k}(z;\z) \right| \le \frac{ (\Gamma_E(\alpha/2) \widetilde C)^k }{\Gamma_E(\alpha k /2)(t-\t)^{1 - k\alpha/2}} 
 \G^+(z;\z),  \qquad k \in \mathbb{N},
\end{equation*}
were $\Gamma_E$ denotes the Euler's Gamma function. Theorem 8 in \cite[Chapter 1]{Friedman} states that, under the assumption that the coefficients $a_{ij}, \frac{\p a_{ij}}{\p x_i}, b_{i}, \frac{\p b_{i}}{\p x_i}$, for $i, j = 1, \dots, N$ and $c$ belong to the space $C^\alpha(\R^N \times ]T_0, T_1[)$ with $T_0 < T_1$ and satisfy \eqref{e-up}, the series  
\begin{equation} \label{eq-Gamma}
 \G (z;\z) := Z(z;\z) + \sum_{k=1}^{\infty} \int_\t^t \bigg( \int_{\R^N}Z(x,t;y,s)  (\L Z)_k (y,s;\x,\t) d y \bigg) d s 
\end{equation}
converges in $\R^N \times ]T_0, T_1[$ and it turns out that its sum $\G$ is a fundamental solution for $\L$. We next list some properties of the function $\G$ defined in \eqref{eq-Gamma}. We mainly refer to Chapter I  in the monograph \cite{Friedman} by Friedman. 
\begin{enumerate} 
\item Theorem 8 in \cite{Friedman}: for every $\z \in \R^{N+1}$ the function $\G(\cdot \, ; \z)$ belongs to $C^{2,1}(\R^N \times ]\t, + \infty[)$ and it is a classical solution to 
$\L \, \G = 0$ in $\R^N \times ]\t, + \infty[$.
\item Theorem 9 in \cite{Friedman}: for every bounded functions $\phi \in C(\R^N)$ and 
$f \in C^\alpha(\R^N \times ]\t, T_1[)$, with $T_0 < \t < T_1$, the function 
$$
  u(z)=\int_{\R^{N}}\Gamma(z;\z)\phi(\x)d \x - \int_\t^t \bigg( \int_{\R^{N}}\Gamma(x,t;\x, s) f(\x, s)d \x \bigg) d s 
$$
is a classical solution to the Cauchy problem
\begin{equation} \label{cauchyproblem}
\left\{
  \begin{array}{ll}
    \L u = f, & \hbox{$z \in \R^{N} \times ]\t, + \infty[$} \\
    u( \cdot,\t) = \varphi & \hbox{in} \ \R^N.
  \end{array}
\right.
\end{equation}
\item Theorem 15 in \cite{Friedman}: The function $\G^*(z;\z) := \G(\z;z)$ is the fundamental solution of the transposed operator $\L^*$ acting on a suitably smooth function $v$ as follows
\begin{equation} \label{e-L*}
    \L^* v (z)  := \div \left( A (z) \nabla_x v(z) \right) -  \langle b(z), \nabla_x v(z)\rangle + 
	(c(z) - \div \, b(z)) v(z) + \, \tfrac{\p u}{\p t}(z).
\end{equation}
\item Inequalities (6.10) and (6.11) in \cite{Friedman}: for every positive $T$ and $\Lambda^+ > \Lambda$ there exists a positive constant $C^{+}$ such that 
\begin{equation} \label{upper-bound}
\G(z; \z) \le C^{+} \, \G^{+} (z; \z),
\end{equation}
for every $z = (x,t), \z = (\x, \t) \in \R^{N+1}$ with $0 < t- \t < T$. 
Moreover, the following bounds for the derivatives hold
\begin{equation} \label{e-deriv-bds}
\begin{split}
\left| \frac{\partial \G}{\partial {x_j}}(z;\z) \right| & \le \frac{C^+}{\sqrt{t-\t}} \G^+(z;\z), \quad
 \left| \frac{\partial^2 \G}{\partial{x_i x_j}}(z;\z) \right|  \le \frac{C^+}{t-\t} \G^+(z;\z), 
 \\
 \left| \frac{\partial \G}{\partial {\x_j}}(z;\z) \right| & \le \frac{C^+}{\sqrt{t-\t}} \G^+(z;\z), \quad
 \left| \frac{\partial^2 \G}{\partial{\x_i \x_j}}(z;\z) \right|  \le \frac{C^+}{t-\t} \G^+(z;\z),
\end{split}
\end{equation}
for any $i,j=1, \dots, N$ and for every $z, \z \in \R^{N+1}$ with $0 < t- \t < T$.
\end{enumerate}
We recall that the monograph \cite{Friedman} also contains an existence and uniqueness result for the Cauchy problem under the assumptions that the functions $\varphi$ and $f$ in the Cauchy problem \eqref{cauchyproblem} do satisfy the following growth condition:
\begin{equation*} 
 | \varphi(x)| + |f(z)| \le C_0 \exp \left( h |x|^2 \right) \qquad \text{for every} \ x\in \R^N \ \text{and} \ t \in ]\tau, T_1],
\end{equation*}
for some positive constants $C_0$ and $h$. The reproduction property \eqref{eq-rep-prop-A} for $\G$ holds as a direct consequence of the uniqueness of the solution to the Cauchy problem. We also have 
\begin{equation*}
 			e^{-\Lambda(t-\t)} \le \int_{\R^{N}} \G(x,t;\x, \t) \; d \x \le e^{\Lambda(t-\t)}
\end{equation*} 
for every $(x, t), (\x, \t) \in \R^{N+1}$ with $\t < t$, where $\Lambda$ is the constant introduced in \eqref{e-up}. 


We conclude this section by quoting a statement on the asymptotic behavior of fundamental solutions, which in the stochastic theory is referred to as \emph{large deviation principle}. In our setting it is useful in the description of the \emph{parabolic ball} $\Omega_r(z_0)$ introduced in \eqref{e-Psi}. The first large deviation theorem is due to Varhadhan \cite{Varadhan1967behavior, Varadhan1967diffusion}, who considers parabolic operators $\L$ whose coefficients only depend on $x$ and are H\"older continuous. It states that
\begin{equation}
  4 (t-\tau) \log (\Gamma(x,t;\xi,\tau)) \longrightarrow - d^2(x,\xi) \quad \text{as} \ t \to \tau,
\end{equation}
uniformly with respect to $x,\xi$ varying on compact sets. Here $d(x,\xi)$ denotes the Riemannian distance (induced by the matrix $A$) of $x$ and $\xi$. Several extensions of the large deviation principle are available in literature, under different assumption on the regularity of the coefficients of $\L$. Azencott considers in \cite{Azencott84} operators with smooth coefficients and proves more accurate estimates for the asymptotic behavior of $\log\big(\Gamma(x,t;\x,\t)\big)$. Garofalo and Lanconelli prove an analogous result by using purely PDEs methods in \cite{GarofaloLanconelli-1989}. We recall here a version of this result which is suitable for our purposes.

\begin{proposition} \label{prop-localestimate}
{\sc [Theorem 1.2 in \cite{Polidoro2}]} \  For every $\eta \in ]0,1[$ there exists $C_\eta>0$ such that
\begin{equation} \label{eq:approximation}
 (1 - \eta) Z(z;\z) \le  \G(z;\z) \le  (1 + \eta) Z(z;\z)
\end{equation}
for every $z, \z \in \R^{N+1}$ such that $Z(z;\z) > C_\eta $.                                                                                               
\end{proposition}

We finally prove a simple consequence of Proposition \ref{prop-localestimate} that will be used in the following. We introduce some further notation in order to give its statement. We first note that the function $\Gamma^*$ can be built by using the parametrix method, starting from the expression of the parametrix relevant to $\L^*$, that is
\begin{equation} \label{eq-fundsol-Z*}
 Z^*(z;\z) := \G^*_{A(\z)}(z;\z) = \frac{1}{\sqrt{(4 \pi (\t-t))^{N}\det A(\z)}} \exp \left( - 
 \frac{\langle A(\z)^{-1}(x-\x),x-\x \rangle}{4(\t-t)} \right).
\end{equation}
We set 
\begin{equation} \label{eq-Omega_r*}
 \Omega_r^*(z_0) := \bigg\{ z \in \R^{N+1} \mid Z^*(z;z_0) \ge \frac{2}{r^N} \bigg\},
\end{equation}
and we point out that its explicit expression is:
\begin{equation} \label{eq-Omega_r*-exp}
\begin{split}
 \Omega_r^*(z_0) = & \Big\{ (x,t) \in \R^{N+1} \mid \langle A^{-1}(z_0)(x-x_0), x-x_0 \rangle \le \\
 & \qquad - 4 (t_0 - t) \left( \log \big( \tfrac{2}{r^N} \big) + \tfrac{1}{2} \log (\text{det} A(z_0) )+ \tfrac{N}{2} \log (4 \pi (t_0 - t))  \right) \Big\}.
\end{split}
\end{equation}
We have

\begin{lemma} \label{lem-localestimate*}
There exists a positive constant $r^*$, only depending on the operator $\L$, such that 
\begin{equation*} 
 \Omega_r^*(z_0) \subset \Omega_r(z_0) \subset \Omega_{3r}^*(z_0)
\end{equation*}
for every $z_0 \in \R^{N+1}$ and $r \in ]0,r^*]$.                                                                                               
\end{lemma}

\begin{proof} 
As said before, the function $\Gamma^*$ can be built by using the parametrix $Z^*$ defined in \eqref{eq-fundsol-Z*}. In particular, Proposition \ref{prop-localestimate} applies to $\Gamma^*$. Then, if we apply the estimate \eqref{eq:approximation} with $\eta = \frac12$ and we use \eqref{e-L*}, we find that there exists $C^*>0$ such that
\begin{equation*} 
 \frac12 Z^*(\z;z_0) \le  \G(z_0;\z) \le  \frac32 Z^*(\z;z_0)
\end{equation*}
for every $z_0, \z \in \R^{N+1}$ such that $Z^*(\z;z_0) > C^*$. The claim then follows from \eqref{e-Psi} and \eqref{eq-Omega_r*} by choosing $r^*:=\left(\frac{2}{C^*}\right)^{1/N}$.
\end{proof}

\medskip

We conclude this section with a further result useful in the proof of the Harnack inequality.

\begin{lemma} \label{lem-localestimategradient} {\sc [Proposition 5.3 in \cite{Polidoro2}]} \  
Let $r^*$ be the constant appearing in Lemma \ref{lem-localestimate*}. There exists a positive constants $C$, only depending on the operator $\L$, such that 
\begin{equation*} 
 \left| \partial_{x_j} \Gamma(z_0, z) \right| \le C \left( \frac{|x_0-x|}{t_0-t} + 1 \right) \Gamma(z_0, z), 
 \qquad j=1, \dots, N,
\end{equation*}
for every $z_0 \in \R^{N+1}$ and $z \in \Omega_r(z_0)$ with $r \in ]0,r^*]$.                                                                                               
\end{lemma}

\setcounter{equation}{0} 
\section{A generalized divergence theorem}\label{SectionDivergence}

Let $\Omega$ be an open subset of $\R^n$, and let $\Phi \in C^1 \pr{\Omega;\R^{n}}$. The classical divergence formula reads
\begin{equation} \label{eq-div}
 \int_{ E } \mathrm{div}\,\Phi\ dz =-\int_{ \p E} \scp{\nu,\Phi}\ d \H^{n-1},
\end{equation}
where $E$ is a bounded set such that $\overline E \subset \Omega$ and its boundary is $C^1$. 

We are interested in the situation in which $E$ is the \emph{super-level} set of a real valued 
function $F \in C^1\pr{\O}$, that is $E = \left\{F>y\right\}$ for some $y \in \R$. At every point 
$z\in \partial E$ such that $\nabla F(z)\neq 0$ the \emph{inner} unit normal vector $\nu = \nu(z)$ 
appearing in \eqref{eq-div} is defined as $\nu(z) = \frac{1}{|\nabla F(z)|}\nabla F(z)$ and $\partial E$ 
is a $C^1$ manifold in a neighborhood of $z$. But, if we denote
$$
    \mathrm{Crit} \pr{F}:=\left\{z \in \R^n : \nabla F =0\right\},
$$
the set of \emph{critical points} and $F \pr{\mathrm{Crit}\pr{F}}$ the set of \emph{critical values} of 
$F$, under our hypotheses we cannot apply the classical Sard theorem to state that 
``for almost every $y \in \R$ the level set $\{F=y\}$ is globally a $C^1$ manifold''. Indeed, Whitney proves in \cite{whitney1935} that there exist functions $F \in C^1\pr{\O}$ having the property that $\{F=y\} \cap \mathrm{Crit} \pr{F}$ is not empty for every $y$. 
Therefore, the purpose of this section is to discuss a version of \eqref{eq-div} 
when the boundary of $E$ is $C^1$ up to a closed set of null Hausdorff measure and to see 
how it can be applied in our framework. We first introduce the class of sets with the relevant 
regularity and state the corresponding divergence formula. 
We draw this definition and the following theorem from \cite[Section 9.3]{Maggi}. 

\begin{definition}\label{def-almost-C1}
An open set $E\subset \R^n$ has {\em almost $C^1$-boundary} if there is a closed set $M_0\subset \partial E$
with $\H^{n-1}(M_0)=0$ such that, for every $z_0\in M = \partial E \setminus M_0$ there exist $s > 0$ and 
$F \in C^1 (B(z_0, s))$ with the property that
\begin{align*}
B(z_0,s) \cap E & = \{z\in B(z_0,s):\ F(z) > 0\} ,
\\
B(z_0,s) \cap \partial E & = \{z\in B(z_0,s):\ F(z) = 0\} 
\end{align*}
and $\nabla F(z)\neq 0$ for every $z\in B(z_0,s)$. 
We call $M$ the {\em regular part} of $\partial E$ (note that $M$ is a $C^1$-hypersurface). 
The inner unit normal to $E$ is the continuous vector field
$\nu\in C^0 (M;{\mathbb S}^{n-1})$ given by 
\[
\nu(z) = \frac{\nabla F(z)}{|\nabla F(z)|}, \quad
z \in B(z_0,s) \cap M .
\]
\end{definition}

Let us state the divergence theorem for sets with almost $C^1$-boundary. 
 
\begin{theorem}\label{gen-div-thm}
If $E\subset\R^n$ is an open set with almost $C^1$-boundary and $M$ is the regular part of its boundary, 
then for every $\Phi\in C^1_c(\R^n;\R^n)$ the following equality holds
\begin{equation}\label{eq-div-almost}
 \int_{ E } \mathrm{div}\,\Phi\ dz =-\int_{M} \scp{\nu,\Phi}\ d \H^{n-1} .
\end{equation}
\end{theorem} 

If $F \in C^1\pr{\O}$ and $E = \left\{F>y\right\}$ for some $y \in \R$, we can apply 
Theorem \ref{gen-div-thm} thanks to the following result due to 
A.~Ya.~Dubovicki\v{\i} \cite{Dubovickii}, that generalizes Sard's theorem.

\begin{theorem}[\sc Dubovicki\v{\i}] \label{th-Dubo}
Assume that $\N^n$ and $\M^m$ are two smooth Riemannian manifolds of dimension $n$ and $m$, respectively. 
Let $F:\N^n \rightarrow \M^m$ be a function of class $C^k$. Set $s=n-m-k+1$, then for 
$\H^m-$a.e.~$y \in \M^m$
\begin{equation}\label{e-Du} 
\H^s \pr{\left\{F=y \right\} \cap \mathrm{Crit} \pr{F}}=0.
\end{equation}
\end{theorem}
Notice that if $m=k=1$ and $\M^m=\R$, then $s=n-1$ and for $\H^1-$a.e.~$y \in \R$ the critical part 
of $\left\{F=y \right\}$ is an $\H^{n-1}$ null set, while its regular part is an $\pr{n-1}-$manifold of 
class $C^1$. In other words, $\{F=y\}$ is a set with  almost $C^1$-boundary and we cannot apply  
the classical divergence theorem \eqref{eq-div}, but rather Theorem \ref{gen-div-thm}. 
Summarizing, we have the following result, that immediately follows from the above discussion.

\begin{proposition} \label{prop-1}
Let $\Omega$ be an open subset of $\R^{n}$ and let $F \in C^1 \pr{{\Omega};\R}$.
Then, for $\H^1$-almost every $y \in \R$, we have:
$$
    \int_{\left\{F>y\right\}} \mathrm{div}\,\Phi\ dz =
    -\int_{\left\{F=y\right\}\setminus \mathrm{Crit} \pr{F}} \scp{\nu,\Phi}\ d \H^{n-1},
    \quad \forall \: \Phi \in C_c^1 \pr{\Omega;\R^{n}},
$$
were $\nu=\tfrac{\nabla F}{\abs{\nabla F}}$.
\end{proposition}
\begin{proof} By Dubovicki\v{\i} Theorem \ref{th-Dubo} for $\H^1-$almost every $y \in \R$ 
the set $\{F>y\}$ has almost $C^1$-boundary, hence Theorem \ref{eq-div-almost} applies. 
Moreover, as $F$ is continuous, for any such $y$ we have $\partial\{F>y\}\subset\{F=y\}$, 
$\H^{n-1}(\{F=y\}\setminus\partial\{F>y\})=0$ and the regular part of $\partial\{F>y\}$ is 
$\{F=y\}\setminus\{\nabla F=0\}$ and has full $\H^{n-1}$ measure.
\end{proof}

\medskip

In order to prove Theorem \ref{th-1} we apply Proposition \ref{prop-1} to the super-level set 
$\Omega_r(z_0)$ of the fundamental solution $\Gamma(z_0,\cdot)$ of $\L$. Then, as explained 
in the Introduction, we have to cut at a time less than $t_0$ to avoid the singularity of the kernels 
at $z_0$. Therefore, we specialize Proposition \ref{prop-1} as follows. 

\begin{proposition}\label{prop-1bis}
Let $G \in C^1 \pr{\R^{N+1} \setminus \left\{ \pr{x_0,t_0} \right\};\R}$. Then for $\H^1-$almost 
every $w, \varepsilon \in \R$ 
$$
\int_{\left\{ G > w \right\} \cap \left\{ t<t_0-\varepsilon \right\}}
\!\!\!\!\!\!\mathrm{div}\Phi\, dz
=-\int_{(\{ G = w \} \setminus \mathrm{Crit}\pr{G}) \cap \{ t<t_0-\varepsilon \}}
\!\!\!\!\!\!\scp{\nu,\Phi}d \H^{N}
+\int_{\left\{ G > w \right\} \cap \left\{ t=t_0-\varepsilon \right\}} 
\!\!\!\!\!\!\scp{e,\Phi}d\H^{N},
$$
for every $\Phi \in C^1_c \pr{\Omega;\R^{N+1}}$, where $\nu=\tfrac{\nabla G }{\abs{\nabla G}}$ and 
$e=\pr{0,\ldots,0,1}$.
\end{proposition}
\begin{proof}
Notice that for $\H^1-$a.e. $w\in\R$ the level set $\{G>w\}$ has almost-$C^1$ boundary and fix 
such a value.
Let $S$ be the $\H^N$-negligible singular set of $\partial\{G>w\}$: by Fubini theorem, for 
$\H^1-$a.e. 
$\varepsilon>0$ the set $S\cap\{t=t_0-\varepsilon\}$ is in turn $\H^{N-1}-$negligible, 
and out of this set the unit normal is given $\H^N-$a.e. by $\nu$ in 
$\{ G = w \} \setminus \mathrm{Crit}\pr{G} \cap \{ t<t_0-\varepsilon \}$ and by $e$ in 
$\{ G > w \} \cap \{ t=t_0-\varepsilon \}$. Therefore, Proposition \ref{prop-1} applies 
with $n=N+1$, $\Omega=\R^{N+1}\setminus \left\{ \pr{x_0,t_0} \right\}$, 
\[
F(x,t)= (G(x,t)- w)\wedge (t-t_0+\varepsilon) ,
\]
$y=0$ and the set 
\begin{equation*}
\Sigma=(\partial\{ G > w \} \cap \{ t<t_0-\varepsilon \} \cap \mathrm{Crit}\pr{G} \bigr)
\cup \bigl(\{ G = w \} \cap \{ t=t_0-\varepsilon \}\bigr)
\end{equation*}
is $\H^N$-negligible.
\end{proof}

\medskip

The last result we need to prove Theorem \ref{th-1} is the coarea formula for Lipschitz functions. 
We refer to \cite[3.2.12]{federer1969geometric} or \cite{AmbrosioFuscoPallara}, Theorem 2.93 and 
formula (2.74) for the proof. 

\begin{theorem}[\sc Coarea formula for Lipschitz functions] \label{th-co}
Let $G:\R^n \rightarrow \R$ be a Lipschitz function, and let $g$ be a non-negative measurable function. Then
\begin{equation} \label{e-co} 
\int_{\R^n} g\pr{z}\abs{\nabla G \pr{z}}dz=
\int_{\R}\pr{\int_{\left\{ G=y \right\}}g\pr{z}d\H^{n-1} \pr{z}}dy.
\end{equation}
\end{theorem}

\setcounter{equation}{0} 
\section{Proof of the mean value formulas and maximum principle}\label{sectionProof}

In this Section we give the proof of the mean value formulas and of the strong maximun principle.

\medskip

\begin{proof} {\sc of Theorem \ref{th-1}.} Let $\Omega$ be an open subset of $\R^{N+1}$, and let $u$ be a classical solution to $\L u = f$ in $\Omega$. Let $z_0=(x_0,t_0) \in \Omega$ and let $r_0>0$ be such that  $\overline{\Omega_{r_0}(z_0)} \subset \Omega$. We prove our claim by applying Proposition \ref{prop-1bis} with $G(z) = \Gamma(z_0; z)$ and $w = \frac{1}{r^N}$, where $r \in ]0,r_0]$ is such that the statement of Proposition \ref{prop-1bis} holds true with $w = \frac{1}{r^N}$, and $\varepsilon := \varepsilon_k$ for some monotone sequence $\big(\varepsilon_k\big)_{k \in \mathbb{N}}$ such that $\varepsilon_k \to 0$ as $k \to + \infty$ (see Figure 3).

\bigskip

\begin{center}
\begin{tikzpicture}
\clip (-.5,7.5) rectangle (6.7,2);
\shadedraw [top color=blue!10] (-2,6) rectangle (7,1); 
\begin{axis}[axis y line=middle, axis x line=middle, 
    xtick=\empty,ytick=\empty, 
    ymin=-1.1, ymax=1.1,
    xmin=-.2,xmax=1.8, samples=101, rotate= -90]
\addplot [black,line width=.7pt, domain=-.01:.01] {sqrt(.0001 - x * x)} node[above] {\hskip12mm $(x_0,t_0)$};
\addplot [black,line width=.7pt, domain=-.01:.01] {-sqrt(.0001 - x * x)};
\addplot [blue,line width=.7pt, domain=.001:1] {sqrt(- 2 * x * ln(x))}; 
\addplot [blue,line width=.7pt,domain=.001:1] {- sqrt(- 2 * x * ln(x))} 
node[below] { \hskip20mm $\Omega_r(x_0,t_0)$};
\end{axis}
\draw [<-,line width=.4pt] (2.8475,7) -- (2.8475,2);
\draw [red, line width=.6pt,] (-1,6) -- (6.7,6) node[below] { \hskip-18mm  $t = t_0 - \varepsilon_k$}; 
\path[draw,thick] (-.49,7.49) rectangle (6.69,2.01);
\end{tikzpicture} 

{\sc Fig.3}  - The set $\Omega_r(x_0,t_0) \cap \big\{ t < t_0 - \varepsilon_k \big\}$.
\end{center}

\bigskip 

For this choice of $r$, we set $v(z) := \Gamma(z_0;z) - \frac{1}{r^N}$, and we note that 
\begin{equation} \label{eq-div-L*}
\begin{split}
 u(z) \L^* v(z) - v(z) \L u(z) = 
 & \div_x \big( u(z) A(z) \nabla_x v(z) - v(z) A(z) \nabla_x u(z) \big) - \\
 & \div_x \big( u(z) v(z) b(z) \big) + \partial_t (u(z)v(z))
\end{split}
\end{equation}
for every $z \in \Omega \backslash \big\{z_0 \big\}.$ We then recall that $\L^* v = \frac{1}{r^N} 
\left( \div \, b - c \right)$ and $\L u  = f$ in $\Omega \backslash \big\{z_0 \big\}$. 
Then \eqref{eq-div-L*} can be written as follows
\begin{equation*}
 \frac{1}{r^N} \left( \div \, b(z) - c(z) \right) u(z) - v(z) f(z) = 
 \div \, \Phi (z), \qquad \Phi(z) := \big( u A \nabla_x v - v A \nabla_x u - uv b, uv \big)(z).
\end{equation*}
We then apply Proposition \ref{prop-1bis} to the set
$\Omega_r(z_0) \cap \left\{ t<t_0-\varepsilon_k \right\}$ and we find
\begin{equation} \label{eq-div-k}
\begin{split}
 \int_{\Omega_r(z_0) \cap \left\{ t<t_0-\varepsilon_k \right\}}
\!\!\!\!\!\!\!\!\!\!\!\!\!\!\!\!\!\!\!\!\!  \left( \tfrac{1}{r^N} \left( \div \, b(z) - c(z) \right) u(z) - v(z) f(z)\right) dz
= & \\ 
- \int_{\psi_r(z_0)\setminus \mathrm{Crit}\pr{\Gamma} \cap \left\{ t<t_0-\varepsilon_k \right\}}
\!\!\scp{\nu,\Phi} & d \H^{N} + \int_{\Omega_r(z_0) \cap \left\{ t=t_0-\varepsilon_k \right\}} 
\!\!\scp{e,\Phi}d\H^{N},
\end{split}
\end{equation}
where $\nu(z)=\tfrac{\nabla_{(x,t)} \Gamma(z_0,z) }{\abs{\nabla_{(x,t)} \Gamma(z_0,z)}}$ and 
$e=\pr{0,\ldots,0,1}$. We next let $k \to + \infty$ in the above identity. As $f$ is continuous on 
$\overline{\Omega_r(z_0)}$ and $v \in L^1({\Omega_r(z_0)})$, we find 
\begin{equation} \label{eq-div-1}
\begin{split}
  \lim_{k \to + \infty} \int_{\Omega_r(z_0) \cap \left\{ t<t_0-\varepsilon_k \right\}}
 &\left( \tfrac{1}{r^N} \left( \div \, b(z) - c(z) \right) u(z) - v(z) f(z)\right) dz = \\
& \int_{\Omega_r(z_0)} \left( \tfrac{1}{r^N} \left( \div \, b(z) - c(z) \right) u(z) - v(z) f(z)\right) dz.
\end{split}
\end{equation}
We next consider the last integral in the right hand side of \eqref{eq-div-k}. We have 
$\scp{e,\Phi} (z) = u(z) v(z)$, then 
\begin{equation} \label{eq-psir}
 \int_{\Omega_r(z_0) \cap \left\{ t=t_0-\varepsilon_k \right\}} \!\!\scp{e,\Phi}d\H^{N} =
 \int_{\mathcal{I}_r^k (z_0)} \!\! \!\! 
 u(x,t_0- \varepsilon_k) \left( \Gamma(x_0, t_0; x,t_0 - \varepsilon_k) - \frac{1}{r^N} \right) d x,
\end{equation}
where we have denoted
\begin{equation*} 
 \mathcal{I}_r^k (z_0) := \left\{ x \in \R^N \mid (x,t_0-\varepsilon_k) \in \overline{\Omega_r (z_0)} \right\}.
\end{equation*}
We next prove that the right hand side of \eqref{eq-psir} tends to $u(z_0)$ as $k \to + \infty$. Since 
$\Gamma$ is the fundamental solution to $\L$ we have
\begin{equation*} 
 \lim_{k \to + \infty}  \int_{\R^N} \!\! \!\! \Gamma(x_0, t_0; x,t_0 - \varepsilon_k)
 u(x,t_0- \varepsilon_k)  d x = u(x_0,t_0), 
\end{equation*}
then, being $u$ continuous on $\overline{\Omega_r (z_0)}$, we only need to show that 
\begin{equation} \label{eq-claim-2}
\lim_{k \to + \infty} \H^{N} \left( \mathcal{I}_r^k (z_0) \right) = 0, \qquad
 \lim_{k \to + \infty}  \int_{\R^N \backslash \mathcal{I}_r^k (z_0)} \!\! \!\! 
 \Gamma(x_0, t_0; x,t_0 - \varepsilon_k) d x = 0.
\end{equation}
With this aim, we note that the upper bound \eqref{upper-bound} and \eqref{eq-fundsol+} imply
\begin{equation*} 
 \mathcal{I}_r^k (z_0) \subset \left\{ x \in \R^N \mid | x - x_0|^2 \le 4 \Lambda^+ \varepsilon_k 
 \left( \log \left(C^+ r^N \right) - \tfrac{N}{2} \log (4 \pi \Lambda^+  \varepsilon_k)  \right) \right\}.
\end{equation*}
The first assertion of \eqref{eq-claim-2} is then a plain consequence of the above inclusion. In order to prove the second statement in \eqref{eq-claim-2}, we rely on Lemma \ref{lem-localestimate*}. We let $r_0 := \min (r, r^*)$, so that
\begin{equation*} 
 \Omega_{r_0}^*(z_0) \subset \Omega_{r_0}(z_0) \subset \Omega_r(z_0),
\end{equation*}
thus 
\begin{equation*} 
\begin{split}
 \R^N \backslash \mathcal{I}_r^k (z_0) \subset & \left\{ x \in \R^N \mid Z^*(x,t_0- \varepsilon_k;x_0,t_0) \le \tfrac{2}{r_0^N} \right\} \\
 = & \Big\{ x \in \R^N \mid \langle A(z_0) (x - x_0), x-x_0 \rangle \\
 & \quad \ge - 4 \varepsilon_k \left( \log \left( \tfrac{2}{r_0^N} \right) + \tfrac{1}{2} \log (\det A(z_0) ) + \tfrac{N}{2} \log (4 \pi \varepsilon_k) \right) \Big\}.
\end{split} 
\end{equation*}
By using again \eqref{upper-bound}, the above inclusion, and the change of variable $x = x_0 + 2 \sqrt{\Lambda^+ \varepsilon_k} \, \xi$, we find
\begin{equation*} 
\begin{split}
  \int_{\R^N \backslash \mathcal{I}_r^k (z_0)} & \!\! \!\! \Gamma(x_0, t_0; x,t_0 - \varepsilon_k) d x \le C^+ \int_{\R^N \backslash \mathcal{I}_r^k (z_0)} \!\! \!\! \Gamma^+ (x_0, t_0; x,t_0 - \varepsilon_k) d x \\ 
  & \le \frac{C^+}{\pi^{N/2}} \int_{ \left\{ \langle A(z_0) \xi, \xi \rangle \ge - \tfrac{1}{\Lambda^+} \left( \log \left( \tfrac{2}{r_0^N} \right) + \tfrac{1}{2} \log (\det A(z_0) ) + \tfrac{N}{2} \log (4 \pi \varepsilon_k) \right) \right\} }
 \!\! \exp\left( - |\xi|^2 \right) d \xi.
\end{split} 
\end{equation*}
The second assertion of \eqref{eq-claim-2} then follows. Thus, we have shown that
\begin{equation} \label{eq-div-2}
 \lim_{k \to + \infty} \int_{\Omega_r(z_0) \cap \left\{ t=t_0-\varepsilon_k \right\}} \scp{e,\Phi} d\H^{N} = u(z_0).
\end{equation}
We are left with the first integral in the right hand side of \eqref{eq-div-k}. We preliminarily note that its limit, as $k \to + \infty$, does  exist. Moreover, for every $z \in \psi_r(z_0)$ we have $v(z) = 0$, then $\Phi(z) = \big( u (z) A(z) \nabla_x v(z), 0 \big)$, so that 
\begin{equation*} 
 \int_{\psi_r(z_0) \setminus \mathrm{Crit}\pr{\Gamma} \cap \left\{ t<t_0-\varepsilon_k \right\}}
\!\!\scp{\nu,\Phi}d \H^{N} = \int_{\psi_r(z_0) \setminus \mathrm{Crit}\pr{\Gamma} \cap \left\{ t<t_0-z\varepsilon_k \right\}} \!\! u(x,t) K (z_0;z) d \H^{N},
\end{equation*}
where 
\begin{equation*}
	K (z_0; z) = \frac{\langle A(z) \nabla_x \Gamma(z_0;z), \nabla_x \Gamma(z_0;z) \rangle }
	{|\nabla_{(x,t)}\Gamma(z_0; z)|}
\end{equation*}
is the kernel defined in \eqref{e-kernels}. Note that $K$ is non-negative and, if we consider the function 
$u = 1$ and we let $k \to + \infty$, we find
\begin{equation*} 
 \lim_{k \to + \infty} \int_{\psi_r(z_0) \setminus \mathrm{Crit}\pr{\Gamma} \cap 
 \left\{ t<t_0-\varepsilon_k \right\}}
\!\! K (z_0; z) d \H^{N} = \int_{\psi_r(z_0)\setminus \mathrm{Crit}\pr{\Gamma}} \!\!\! 
K (z_0;z) d \H^{N} < + \infty.
\end{equation*}
Thus, if $u$ is a classical solution to $\L u = 0$, we obtain
\begin{equation} \label{eq-div-3i}
 \lim_{k \to + \infty} \int_{\psi_r(z_0) \setminus \mathrm{Crit}\pr{\Gamma} \cap 
 \left\{ t<t_0-\varepsilon_k \right\}}
\!\!\scp{\nu,\Phi}d \H^{N} = \int_{\psi_r(z_0)\setminus \mathrm{Crit}\pr{\Gamma}} \!\!\! 
K (z_0;z) u(z)  d \H^{N}.
\end{equation} 
We recall that Dubovicki\v{\i}'s theorem implies that 
$\H^{N} \left(\psi_r(z_0) \cap \mathrm{Crit}\pr{\Gamma}\right) = 0$ for $\H^{1}$ almost every 
$r$, so that we can equivalently write
\begin{equation} \label{eq-div-3}
 \lim_{k \to + \infty} \int_{\psi_r(z_0) \setminus \mathrm{Crit}\pr{\Gamma} \cap 
 \left\{ t<t_0-\varepsilon_k \right\}}
\!\!\scp{\nu,\Phi}d \H^{N} = \int_{\psi_r(z_0)} \!\!\! K (z_0;z) u(z)  d \H^{N}.
\end{equation} 
The proof of the first assertion of Theorem \ref{th-1} then follows by using \eqref{eq-div-1}, \eqref{eq-div-2} and \eqref{eq-div-3} in \eqref{eq-div-k}. 

The proof of the second assertion of Theorem \ref{th-1} is a direct consequence of the first one and of the 
coarea formula stated in Theorem \ref{th-co}. Indeed, fix a positive $r$ as above, multiply by $\frac{N}{r^N}$ 
and integrate over $]0,r[$. We find
\begin{equation} \label{e-meanvalue-step1}
\begin{split}
	\frac{N}{r^N} \int_0^r \varrho^{N-1} u(z_0) d \varrho = & 
	\frac{N}{r^N} \int_0^r \varrho^{N-1} \bigg(\int_{\psi_\varrho(z_0)} 
	K (z_0;z) u(z) \, d \H^{N} (x,t) \bigg) d \varrho\, 
	\\
	& + \frac{N}{r^N} \int_0^r \varrho^{N-1} \bigg( \int_{\Omega_\r(z_0)} f (z) 
	\left( \tfrac{1}{r^N} - \Gamma(z_0;z) \right) dz \bigg) d \varrho  
	\\
	& + \frac{N}{r^N} \int_0^r \frac{1}{\varrho} \bigg( 
	\int_{\Omega_\r (z_0)} \left( \div \, b(z) - c(z) \right) u(z) \, dz \bigg) d \varrho.
\end{split}
\end{equation}
The left hand side of the above equality equals $u(z_0)$, while the last two terms agree with the last two terms appearing in the statement of Theorem \ref{th-1}. In order to conclude the proof we only need to show that 
\begin{equation} \label{e-meanvalue-step2}
\begin{split}
	  \int_0^r \varrho^{N-1}  \bigg(\int_{\left\{\Gamma(z_0;z) = 
	  \tfrac{1}{\varrho^N}\right\} } & K (z_0;z) u(z) \, d \H^{N} (z) \bigg) d \varrho 
	  \\
	& = \frac{1}{N} \int_0^{r} \r^{N-1}\int_{\Omega_\r (z_0)} M (z_0;z) u(z) dz.
\end{split}
\end{equation}
With this aim, we substitute $y = \frac{1}{\varrho^N}$ in the left hand side of \eqref{e-meanvalue-step2} and we recall the definition of the kernel $K$. We find 
\begin{equation} \label{e-meanvalue-step3}
\begin{split}
	  \int_0^r \varrho^{N-1} & \bigg( \int_{\left\{\Gamma(z_0;z) = \tfrac{1}{\varrho^N}\right\} } 
	  \frac{\langle A(z) \nabla_x \Gamma(z_0;z), \nabla_x \Gamma(z_0;z) \rangle }
	{|\nabla_{(x,t)}\Gamma(z_0;z)|} u(z) \, d \H^{N} (z) \bigg) d \varrho \\
     & = \frac{1}{N}
    \int_{\frac{1}{r^N}}^{+ \infty} \frac{1}{y^{2}} \bigg(\int_{\left\{\Gamma(z_0;z) = y\right\} } 
	\frac{\langle A(z) \nabla_x \Gamma(z_0; z), \nabla_x \Gamma(z_0;z) \rangle }
	{|\nabla_{(x,t)}\Gamma(z_0;z)|}u(z) \, d \H^{N} (z) \bigg) d y \\
	& =\frac{1}{N} \int_{\frac{1}{r^N}}^{+ \infty} \bigg(\int_{\left\{\Gamma(z_0;z) = y\right\} } 
	\frac{\langle A(z) \nabla_x \Gamma(z_0;z), \nabla_x \Gamma(z_0;z) \rangle }
	{\Gamma^2(z_0;z) {|\nabla_{(x,t)}\Gamma(z_0;z)|}} u(z) \, d \H^{N} (z) \bigg) d y.
\end{split}
\end{equation}
We conclude the proof of \eqref{e-meanvalue-step2} by applying the coarea formula stated in Theorem 
\ref{th-co}.
\end{proof}

\medskip

\begin{proof} {\sc of Proposition \ref{prop-smp}.}  We prove our claim under the additional assumption $\div \, b \ge 0$. At the end of the proof we show that this assumption is not restrictive. 

We first note that, as a direct consequence of our assumption $c=0$, we have that $\L \, 1 = 0$, then Theorem \ref{th-1} yields
 \begin{equation*} 
	\frac{1}{\varrho^N} \int_{\Omega_\varrho(z_1)} M (z_1; z) \, dz + 
	\frac{N}{\r^N} \int_0^{\r} \Big(\frac{1}{s} \int_{\Omega_s (z_1)} \div \, b(z) \, dz \Big) d s = 1 
\end{equation*}
for every $z_1 \in \Omega$ and $\r >0$ such that $\overline{\Omega_\varrho(z_1)} \subset \Omega$.

We claim that, if $u(z_1) = \max_\Omega u$, then
\begin{equation} \label{eq-claim-smp}
 u(z) = u(z_1) \qquad \text{for every} \quad z \in \overline{\Omega_\varrho(z_1)}.
\end{equation}
By using again Theorem \ref{th-1} and the above identity we obtain
 \begin{align*} 
	0 = & \frac{1}{\varrho^N} \int_{\Omega_\varrho(z_1)} M (z_1; z) \big((u(z)- u(z_1)\big) \, dz 
	\\
	& + \frac{N}{\r^N} \int_0^{\r} \Big(\frac{1}{s} \int_{\Omega_s (z_1)} \div \, b(z) 
	\big((u(z)- u(z_1)\big) \, dz \Big) d s 
	\\
	& + \frac{N}{\r^N} \int_0^{\r} \left(s^{N-1}
	\int_{\Omega_s (z_1)} f (z) \left( \tfrac{1}{s^N} - \Gamma(z_1;z) \right) dz \right) d s \le 0,
\end{align*}
since $f \ge 0$, $\div \, b \ge 0$ and $u(z) \le u(z_1)$, being $u(z_1) = \max_{\Omega} u$. We have also used the fact that $M(z_1;z) \ge 0$ and $\Gamma(z_1;z) \ge \tfrac{1}{s^N}$ for every $z \in \Omega_s(z_1)$. Hence, $M (z_1; z) \big((u(z)- u(z_1)\big)=0$ for $\H^{N+1}$ almost every $z \in \Omega_\r(z_1)$. As already noticed, Dubovicki\v{\i}'s theorem implies that $\H^{N} \left( \psi_s (z_1) \cap \mathrm{Crit}\pr{\Gamma} \right) = 0$, for almost every $s \in ]0,\r]$, then $M (z_1; z) \ne 0$ for $\H^{N+1}$ almost every $z \in \Omega_\r(z_1)$. As a consequence $u(z)= u(z_1)$ for $\H^{N+1}$ almost every $z \in \Omega_\r(z_1)$, and the claim \eqref{eq-claim-smp} follows from the continuity of $u$.

We are in position to conclude the proof of Proposition \ref{prop-smp}. Let $z$ be a point of $\AS  ( \O )$, and let $\gamma: [0,T] \to \O$ be an $\L$--admissible path such that $\g(0)= z_0$ and $\g(T) = z$. We will prove that $u(\gamma(t)) = u(z_0)$ for every $t \in [0,T]$. Let
\begin{equation*}
 I := \big\{ t \in [0,T] \mid u(\gamma(s)) = u(z_0) \ \text{for every} \ s \in[0,t] \big\}, \qquad \overline t := \sup I.
\end{equation*}
Clearly, $I \ne \emptyset$ as $0 \in I$. Moreover $I$ is closed, because of the continuity of $u$ and $\gamma$, then $\overline t \in I$. We now prove by contradiction that $\overline t = T$. 

Suppose that $\overline t < T$. Let $z_1 := \gamma(\overline t)$ and note that $z_1 \in \Omega$, $u(z_1) = \max_\Omega u$. We aim to show that there exist positive constants $r_1$ and $s_1$ such that $\overline{\Omega_{r_1}(z_1)} \subset \O$ and 
\begin{equation} \label{eq-claim2-smp}
 \gamma(\overline t + s) \in \Omega_{r_1}(z_1)\quad \text{for every} \quad  s \in [0, s_1[. 
\end{equation}
As a consequence of \eqref{eq-claim-smp} we obtain $u(\gamma(\overline t + s)) = u(z_1) = u(z_0)$ for every $s \in [0, s_1[$, and this contradicts the assumption $\overline t < T$. 

The proof of \eqref{eq-claim2-smp} is a consequence of Lemma \ref{lem-localestimate*}. It is not restrictive to assume that  $r_1 \le r^*$, then it is sufficient to show that there exists a positive $s_1$ such that 
\begin{equation} \label{eq-claim4-smp}
 \gamma(\overline t + s) \in \Omega^*_{r_1}(z_1) \quad \text{for every} \quad  s \in [0, s_1[. 
\end{equation}
Recall the definition of $\gamma(\overline t + s) = (x(\overline t + s), t(\overline t + s))$. We have $\gamma(\overline t) = z_1 = (x_1, t_1), t(\overline t + s) = t_1-s$ and, for every positive $s$ 
\begin{equation*}
\begin{split}
 |x(s + \overline t) - x_1 | & = \left| \int_0^s \dot x(\overline t + \sigma) d \sigma \right| \le \int_0^s \left| \dot x(\overline t + \sigma) \right| d \sigma \\
 & \le \left( \int_0^s \left| \dot x(\overline t + \sigma) \right|^2 d\sigma \right)^{1/2} s^{1/2} \le \|\dot x \|_{L^2([0,T])} \sqrt{s},
\end{split}
\end{equation*}
then
\begin{equation*}
\langle A^{-1}(z_1 )(x(\overline t + s)-x_1), x(\overline t + s)-x_1 \rangle \le 
s \cdot \| A^{-1}(z_1)\| \cdot \|\dot x \|_{L^2([0,T])}^2.
\end{equation*}
By using the above inequality in \eqref{eq-Omega_r*-exp} we see that there exists a positive constant $s_1$ such that \eqref{eq-claim4-smp} holds. This proves \eqref{eq-claim2-smp}, and then $u(z) = u(z_0)$ for every $z \in \AS  ( \O )$. By the continuity of $u$ we conclude that $u(z) = u(z_0)$ for every $z \in \overline{\AS  ( \O )}$. Eventually, since $u$ is constant in $\overline{\AS  ( \O )}$ and $c= 0$, we conclude that $\L u = 0$.

We finally prove that the additional assumption $\div \, b \ge 0$ is not restrictive. Let $k$ be any given constant such that $k>\Lambda$, where $\Lambda$ is the quantity appearing in \eqref{e-up}, recall that $z_0 = (x_0,t_0)$ and define the function
\begin{equation}
 v(y,t):= u\left( e^{-k(t-t_0)}y, t\right), \qquad (y,t) \in \widehat \Omega,
\end{equation}
where $(y,t) \in \widehat \Omega$ if, and only if $\left( e^{-k(t-t_0)}y, t\right) \in \Omega$. Then $v$ is a solution to 
\begin{equation*} 
	\widehat \L v (y,t)  := 
	\div \left( \widehat A (y,t) \nabla_y v(y,t) \right) + \langle \widehat b(y,t) + k y, \nabla_y v(y,t)\rangle - 
	\, \tfrac{\p v}{\p t}(y,t) = f\left( e^{-k(t-t_0)}y, t\right),
\end{equation*}
where $\widehat A (y,t) = \left( \widehat a_{ij} (y,t) \right)_{i,j=1, \dots, N}$, $\widehat b(y,t) = \left( \widehat b_1(y,t), \dots, \widehat b_N(y,t) \right)$, are defined as $\widehat a_{ij} (y,t) = e^{-2k(t-t_0)} a_{ij} \left( e^{-k(t-t_0)}y, t\right), \widehat b_{j} (y,t) = e^{-k(t-t_0)} b_{j} \left( e^{-k(t-t_0)}y, t\right)$, for $i,j=1, \dots, N$.  Note that from the assumption \eqref{e-up} it follows that $|\div \, b| \le N \Lambda$, then 
\begin{equation*} 
	\div \left( \widehat b(y,t) + k y \right) \ge N \left(k - \Lambda e^{-2k(t-t_0)} \right).
\end{equation*}
In particular, there exists a positive $\delta$, depending on $k$ and $\Lambda$, such that the right hand side of above expression is non-negative as $t \ge t_0 - \delta$. Then, if we set $\widehat t_0 := t_0- \delta$, we have  
\begin{equation*} 
	\div \left( \widehat b(y,t) + k y \right) \ge 0 \qquad \text{for every} \quad (y,t) \in \widehat \Omega 
	\cap \big\{t \ge \widehat t_0 \big\}.
\end{equation*} 
We also note that $\widehat \L$ satisfies the same assumptions as $\L$, with a possibily different constant $\widehat \Lambda$, in every set of the form $\widehat \Omega \cap \big( \R^N \times I \big)$, where $I$ is any bounded open interval of $\R$. We then apply the above argument to prove that, if $v$ reaches its maximum at some point $(y_0,t_0) \in \widehat \Omega$, then it is constant in its propagation set in $\widehat \Omega \cap \big\{t \ge \widehat t_0 \big\}$. Note that $u$ reaches its maximum at some point $(x,t)$ if, and only if, $v$ reaches its maximum at $\left( e^{k(t-t_0)}x, t\right)$. Moreover, $(x(s), t-s)$ is an admissible curve for $\L$ if, and only if, $(e^{-k(t-s-t_0)}x(s), t-s)$ is an admissible curve for $\widehat \L$. We conclude that
\begin{equation*}
u(z) = u(z_0) \qquad \text{for every} \ z \in \overline {\AS  ( \O )} \cap \big\{t \ge \widehat t_0 \big\}.
\end{equation*}
We then repeat the above argument. Assume that $u$ reaches its maximum at some point $\left( \widehat x_0, \widehat t_0\right)$, we define a new function $\widehat v(y,t):= u\left( e^{-k(t-\hat t_0)}y, t\right)$ and we find a new constant $\widehat t_1 := t_0 - 2 \delta$ such that 
\begin{equation*}
u(z) = u(z_0) \qquad \text{for every} \ z \in \overline {\AS  ( \O )} \cap \big\{t \ge \widehat t_1 \big\}.
\end{equation*}
As we can use the same constant $\delta$ at every iteration, we conclude that the above identity holds for every $z \in \overline {\AS  ( \O )}$.
\end{proof}

\medskip

\begin{proof} {\sc of Proposition \ref{prop-smp1}.}  Let $k$ be a constant such that $\div \, b - c -k \ge 0$ and note that the function $v(x,t) := e^{k t}u(x,t)$ is a non negative solution to the equation 
\begin{equation*}
 \L v(x,t) + k v(x,t) =  e^{k t}f(x,t), \qquad (x,t) \in \Omega.
\end{equation*}
Then Theorem \ref{th-1} yields

 \begin{align*} 
	0 = & \frac{1}{\varrho^N} \int_{\Omega_\varrho(z_1)} M (z_1; z) v(z) \, dz 
	\\
	& + \frac{N}{\r^N} \int_0^{\r} \Big(\frac{1}{s} \int_{\Omega_s (z_1)} \big(\div \, b(z) - c(z) - k \big) 
	v(z) \, dz \Big) d s 
	\\
	& + \frac{N}{\r^N} \int_0^{\r} \left(s^{N-1}
	\int_{\Omega_s (z_1)} e^{k t} f (z) \left( \tfrac{1}{s^N} - \Gamma(z_1;z) \right) dz \right) d s \le 0,
\end{align*}
for every $z_1 \in \Omega$ and $\r >0$ such that $\overline{\Omega_\varrho(z_1)} \subset \Omega$. Here we have sued the facts that
$f \ge 0$, and $\div \, b - c -k \ge 0$. By following the same argument used in the proof of Proposition \ref{prop-smp} we find that $v \ge 0$ in $\overline {\AS  ( \O )}$. This concludes the proof of Proposition \ref{prop-smp1}.
\end{proof}

\medskip

\begin{proof} {\sc of Proposition \ref{prop-2}.}
Let $m$ be a positive integer, and let $u$ be a solution to $\L u = f$ in $\Omega \subset \R^{N+1}$. As said in the Introduction, we set
\begin{equation*}
 \widetilde u(x,y,t) := u(x,t), \qquad \widetilde f(x,y,t) := f(x,t), 
\end{equation*}
for every $(x,y,t) \in \R^{N}\times \R^{m} \times \R$ such that $(x,t) \in \Omega$, and we note that 
\begin{equation*}
 \widetilde \L \ \widetilde u(x,y,t) = \widetilde f(x,y,t) 
 \qquad \widetilde \L  := \L + \sum_{j=1}^{m}\tfrac{\partial^2}{\partial y_j^2}.
\end{equation*}
Moreover, the function
\begin{equation} \label{eq-tildegamma}
 \widetilde \Gamma(x_0, y_0, t_0 ;x,y,t) := \Gamma (x_0, t_0 ;x,t) \cdot
 \frac{1}{(4 \pi (t_0-t))^{m/2}}\exp \left( \frac{-|y_0-y|^2}{4(t_0-t)} \right)
\end{equation}
is a fundamental solution of $\widetilde \L$. We then use $\widetilde \Gamma$ to represent the solution $u$ in accordance with Theorem \ref{th-1} as follows
\begin{equation*}
\begin{split}
 u(z_0) = & \, \widetilde u(x_0,y_0,t_0) =
\frac{1}{r^{N+m}} \int_{\widetilde\Omega_r(x_0,y_0,t_0)} \! \! \! \! 
\widetilde M(x_0,y_0,t_0; x,y,t) u(x,t) \, dx \, dy\,  dt  \\
 & +\frac{N+m}{r^{N+m}} \int_0^{r} \left(\r^{N+m-1}
	\int_{\widetilde\Omega_\r(x_0,y_0,t_0)} \! \! \! \! 
	f (x,t) \left( \tfrac{1}{\r^{N+m}} - \widetilde \Gamma(x_0,y_0,t_0;x,y,t) \right) \, dx \, dy\,  dt \right) d \r \\
 & + \frac{N+m}{r^{N+m}} \int_0^{r} \left(\frac{1}{\r}
	\int_{\widetilde\Omega_\r(x_0,y_0,t_0)} \! \! \! \! 
	\left( \div \, b(x,t) - c(x,t) \right) u(x,t) \, dx \, dy\,  dt \right) d \r.
\end{split}
\end{equation*}
where $\widetilde\Omega_r(x_0,y_0,t_0)$ is the parabolic ball relevant to $\widetilde \Gamma$ and
\begin{equation*} 
    \widetilde M(x_0,y_0,t_0; x,y,t) = M(x_0,t_0;x,t) + \frac{|y_0-y|^2}{4(t_0-t)^2}. 
\end{equation*}
The proof is accomplished by integrating the above identity with respect to the variable $y$.
\end{proof}

\setcounter{equation}{0} 
\section{Proof of the Harnack inequalities}\label{sectionHarnack}

In this Section we use the mean value formula stated in Proposition \ref{prop-2} to give a simple proof of the parabolic Harnack inequality.

\medskip

\begin{proof} {\sc of Proposition \ref{prop-Harnack}.} We first prove our claim under the additional assumption that $\div \, b - c = 0$. This assumption simplifies the proof as in this case we only need to use the first integral in the representation formula given in Proposition \ref{prop-2}. It will be removed at the end of the proof.

Let $m \in \mathbb{N}$ with $m > 2$, let $\Omega$ be an open subset of $\R^{N+1}, z_0 \in \Omega$ and $r>0$ such that $\Omega_{4r}^{(m)}(z_0) \subset \Omega$. We claim that there exist four positive constants $r_0, \vartheta, M^+, m^-$ such that the following assertions hold for every $r \in ]0,r_0]$.
\begin{description}                                                                                                                                                                                                                                                                                                                                                                                                              
\item[{\it i)}] $K^{(m)}_{r}(z_0) \ne \emptyset$;
\item[{\it ii)}] $M_{\vartheta r}^{(m)} (z; \zeta) \le M^+$ for every $\zeta \in \Omega_{\vartheta r}^{(m)}(z)$;
\item[{\it iii)}] $\Omega_{\vartheta r}^{(m)}(z) \subset \Omega_{4r}^{(m)}(z_0) \cap \big\{\tau \le t_0 - \frac{r^2}{4 \pi \lambda^{N/(N+m)}}\big\}$ for every $z \in K^{(m)}_{r}(z_0)$;
\item[{\it iv)}] $M_{5 r}^{(m)} (z_0; \zeta) \ge m^-$ for every $\zeta = (\xi, \tau) \in \Omega_{4 r}^{(m)}(z_0)$ such that $\tau \le t_0 - \frac{r^2}{4 \pi \lambda^{N/(N+m)}}$.                                                                                                                                                                                                                                    \end{description}
By using Proposition \ref{prop-2} and the above claim it follows that, for every $z \in K^{(m)}_{r}(z_0)$, it holds
\begin{equation} \label{e-core-h}
\begin{split}
 u(z) = & \frac{1}{(\vartheta r)^{N+m}} \int_{\Omega^{(m)}_{\vartheta r}(z)} 
 \! \! \! \! \! \! M_{\vartheta r}^{(m)} (z; \zeta) u(\zeta) \, d\zeta  \\
 & (\text{by {\it ii}}) \le \frac{M^+}{(\vartheta r)^{N+m}} 
 \int_{\Omega^{(m)}_{\vartheta r}(z)} \!\!\! u(\zeta) \, d\zeta \\
& (\text{by {\it iii}}) \le \frac{M^+}{(\vartheta r)^{N+m}} 
 \int_{\Omega_{4r}^{(m)}(z_0) \cap \big\{\tau \le t_0 - \frac{r^2}{4 \pi \lambda^{N/(N+m)}}\big\}} 
 \!\!\! u(\zeta) \, d\zeta \\
 & (\text{by {\it iv}}) \le \frac{M^+}{m^- (\vartheta r)^{N+m}} 
 \int_{\Omega_{5r}^{(m)}(z_0)} \! \! \! \! \! \! M_{5 r}^{(m)} (z_0; \zeta)u(\zeta) \, d\zeta = 
 \frac{5^{N+m}M^+}{\vartheta^{N+m} m^-} u(z_0).
\end{split}
\end{equation}
This proves Proposition \ref{prop-Harnack} with $C_K := \frac{5^{N+m}M^+}{ \vartheta^{N+m}m^-}$.

\medskip

We are left with the proof of our claims. We mainly rely on Lemma \ref{lem-localestimate*}, applied to the function $\widetilde \Gamma$ introduced in \eqref{eq-tildegamma}. In the sequel we let $r^*$ be the constant appearing in Lemma \ref{lem-localestimate*} and relative to $\widetilde \Gamma$, and in accordance with \eqref{eq-Omega_r*}, 
\begin{equation} \label{eq-Omega_rm*}
 \Omega_r^{(m)*}(z_0) := \bigg\{ z \in \R^{N+1} \mid (4 \pi (t_0-t))^{-m/2}Z^*(z;z_0) \ge \frac{2}{r^{N+m}} \bigg\}.
\end{equation}
Moreover, we choose $r_0 := r^*/2$. 

\begin{center}
\begin{tikzpicture}
\clip (-.52,7.02) rectangle (6.52,1.78);
\path[draw,thick] (-.5,7) rectangle (6.5,1.8);
\begin{axis}[axis y line=none, axis x line=none, 
    xtick=\empty,ytick=\empty, 
    ymin=-1.1, ymax=1.1,
    xmin=-.2,xmax=1.8, samples=101, rotate= -90]
\addplot [black,line width=1pt, domain=-.01:.01] {sqrt(.0001 - x * x)} node[above] {$z_0$};
\addplot [black,line width=1pt, domain=-.01:.01] {-sqrt(.0001 - x * x)};
\addplot [lblue,line width=.7pt,domain=.001:.051] {sqrt(- 3 * x * ln(9*x)};
\addplot [lblue,line width=.7pt,domain=.001:.051] {- sqrt(- 3 * x * ln(9*x))};
\addplot [blue,line width=.7pt, domain=.001:1] {sqrt(- 2 * x * ln(x))}; 
\addplot [blue,line width=.7pt,domain=.001:1] {- sqrt(- 2 * x * ln(x))} node[below=-10pt] { \hskip35mm $\Omega^{(m)}_{4r}(z_0)$};
\addplot [ddddblue,line width=.7pt,domain=.051:.1111] {sqrt(- 3 * x * ln(9*x)};
\addplot [ddddblue,line width=.7pt,domain=.051:.1111] {- sqrt(- 3 * x * ln(9*x))};
\addplot [red,line width=.7pt,domain=.0001:.0625,below=15pt,right=9pt] {sqrt(- 4 * x * ln(16*x)};
\addplot [red,line width=.7pt,domain=.0001:.0625,below=15pt,right=9pt] {- sqrt(- 4 * x * ln(16*x))};
\addplot [black,line width=1pt, domain=-.01:.01,below=15pt,right=9pt] {sqrt(.0001 - x * x)} node[below=-3pt] {$z$};
\addplot [black,line width=1pt, domain=-.01:.01,below=15pt,right=9pt] {-sqrt(.0001 - x * x)} node[below=15pt,right=10pt] {\hskip-20mm \color{red} $\Omega^{(m)}_{\vartheta r}(z)$};
\end{axis}
\draw [ddddblue,line width=.7pt] (1.92,6) -- node [below=5pt] { \hskip25mm $K^{(m)}_r(z_0)$} (3.75,6);
\end{tikzpicture} 

{\sc \qquad Fig.4}  - The inclusion \emph{(iii)}.
\end{center}

\medskip

\noindent {\it Proof of i)} \ By the definition \eqref{e-Kr} of $K^{(m)}_{r}(z_0)$ we only need to show that there exists at least a point $(x,t) \in \overline{\Omega^{(m)}_{r}(z_0)}$ with $t \le t_0 - \frac{1}{4 \pi \lambda^{N/(N+m)}} \,r^2$. From Lemma \ref{lem-localestimate*} it follows that $\Omega_r^{(m)*}(z_0) \subset \Omega^{(m)}_r(z_0)$, then we only need to show that that the point $\big(x_0, t_0 - \frac{1}{4 \pi \lambda^{N/(N+m)}} \,r^2\big)$ belongs to $\overline{\Omega^{(m)}_{r}(z_0)}$. In view of \eqref{eq-fundsol-Z*}, this is equivalent to $\det A(z_0)\ge \lambda^N/4$, which directly follows from the parabolicity assumption \eqref{e-up}.

\medskip

\noindent {\it Proof of ii)} \ We first note that \eqref{eq-fundsol+} and the defintion of $N_{\vartheta r}$ directly give
\begin{equation} \label{eq-boundN}
	N_{\vartheta r} (z;\z) \le 2 \sqrt{t-\tau}\sqrt{\log\left(
	\tfrac{C^+ (\vartheta r)^{N+m}}{ (\Lambda^+)^{N/2}(4 \pi (t-\tau))^{(N+m)/2}} \right)} = 2 \sqrt{t-\tau}
	\sqrt{C_1 + \tfrac{N+m}{2}\log\left( \tfrac{\vartheta^2 \, r^2 }{t-\tau} \right)},
\end{equation}
where $C_1$ is a positive constant that only depends on $\L$. Moreover, Lemma \ref{lem-localestimategradient} implies that there exists a positive constant $M_0$, only depending on the operator $\L$, such that 
\begin{equation*} 
 M(z,\z) \le M_0 \left( \frac{|x-\xi|^2}{(t-\tau)^2} + 1 \right), \qquad \text{for every} \
 \zeta \in \Omega_{\vartheta r}^{(m)}(z).
\end{equation*}
Moreover, Lemma \ref{lem-localestimate*} implies that there exists another positive constant $M_1$ such that
\begin{equation} \label{eq-boundM}
 M(z,\z) \le M_1 \left( \frac{1}{t-\tau} \log\left( \frac{\vartheta^2 \, r^2}{t-\tau} \right) + 1 \right), 
 \qquad \text{for every} \
 \zeta \in \Omega_{\vartheta r}^{(m)}(z).
\end{equation}
We point out that the constants $C_1$ and $M_1$ depend neither on the choice of $\vartheta \in ]0,1[$, that will be specified in the following proof of the point {\it iii)}, nor on the choice of $r \in ]0,r_0[$. By using \eqref{eq-boundN} and \eqref{eq-boundM} we conclude that there exists a positive constant $M_2$, that only depends on $\L$ and on $m$, such that
\begin{equation*} 
  M_{\vartheta r}^{(m)} (z; \zeta) \le M_2 (t-\tau)^{m/2} 
  \left( 1 + \left| \log\left( \tfrac{\vartheta^2 \, r^2 }{t-\tau} \right) \right| \right)^{m/2} 
  \left( 1 + \tfrac{1}{t-\tau} \left| \log\left( \tfrac{\vartheta^2 \, r^2 }{t-\tau} \right) \right| \right) 
\end{equation*}
for every $\zeta \in \Omega_{\vartheta r}^{(m)}(z)$. The right hand side of the above inequality is bounded whenever $m>2$, uniformly with respect to $r \in ]0,r_0[$ and $\vartheta \in ]0,1[$. This concludes the proof of {\it ii)}.

\medskip

\noindent {\it Proof of iii)} \ We prove the existence of a constant $\vartheta \in ]0,1[$ as claimed by using a compactness argument and the parabolic scaling. We first observe  
that Lemma \ref{lem-localestimate*} implies that 
\begin{equation*} 
 K_r^{(m)}(z_0)  \subset \overline{\Omega^{(m)*}_{3r}(z_0)} \cap 
 \Big\{ t \le t_0 - \tfrac{1}{4 \pi \lambda^{N/(N+m)}} \,r^2 \Big\},
\end{equation*}
which is a compact subset of $\Omega^{(m)*}_{4r}(z_0)$. We now show that there exists $\vartheta \in ]0,1[$ such that
\begin{equation} \label{eq-claim*}
 \Omega_{3 \vartheta r}^{(m)*}(z) \subset \Omega_{4r}^{(m)*}(z_0) 
 \cap \Big\{\tau \le t_0 - \tfrac{r^2}{4 \pi \lambda^{N/(N+m)}}\Big\}
\end{equation}
for every $z \in \overline{\Omega^{(m)*}_{3r}(z_0)} \cap \Big\{ t \le t_0 - \frac{1}{4 \pi \lambda^{N/(N+m)}} \,r^2 \Big\}$. Our claim {\it iii)} will follow from \eqref{eq-claim*} and from Lemma \ref{lem-localestimate*}. 

We next prove \eqref{eq-claim*} by using the parabolic scaling. We note that
\begin{equation} \label{eq-parabolicscaling}
 (x_0 + r \xi,t_0 + r^2 \tau) \in {\Omega^{(m)*}_{k r}(z_0)} \quad \Longleftrightarrow \quad 
 (x_0+\xi,t_0 + \tau) \in {\Omega^{(m)*}_{k}(z_0)},
\end{equation}
for every positive $k$. We will need to use $k = 3, 4$ and $\vartheta$. 

We next show that \eqref{eq-claim*} holds for $r=1$.  The result for every $r \in ]0,r_0[$ will follow from \eqref{eq-parabolicscaling}. Let $\delta(z_0)$ be the distance of the compact set $\overline{\Omega_{3}^{(m)*}(z_0)} \cap \big\{t \le t_0 - \frac{1}{4 \pi \lambda^{N/(N+m)}}\big\}$ from the boundary of $\Omega_{4}^{(m)*}(z_0)$. We have that $\delta$ is a strictly positive function which depends continuously on $z_0$ through the coefficients of the matrix $A(z_0)$. Moreover, the condition \eqref{e-up} is satisfied, then there exists a positive constant $\delta_0$, only depending on $\lambda, \Lambda$ and $N$, such that 
\begin{equation*}
 \delta(z_0) \ge \delta_0 \quad \text{for every} \ z_0 \in \Omega.
\end{equation*}
On the other hand, the diameter of the set $\Omega^{(m)*}_{1}(z)$ is bounded by a constant that doesn't depend on $z$. Then, by \eqref{eq-parabolicscaling}, it is possible to find $\vartheta \in ]0,1[$ such that the diameter of $\Omega^{(m)*}_{\vartheta}(z)$ is not greater than $\delta_0$. This concludes the proof of \eqref{eq-claim*} in the case $r=1$. As said above, the case $r \in ]0,r_0[$ does follow from \eqref{eq-parabolicscaling}. This concludes the proof of {\it iii)}.

\medskip

\noindent {\it Proof of iv)} \ From \eqref{e-Omegam} it directly follows that
\begin{equation*}
	M_{5r}^{(m)} (z_0;\z) \ge \frac{m \, \omega_m }{m+2} \cdot \frac {N_{5r}^{m+2}(z_0;\z)}{4(t_0-\t)^2} \ge 
	\frac{m \, \omega_m}{m+2} (2(t_0-\t))^{m-2} \left( (N+m) \log(5/4)\right)^{(m+2)/2}.
\end{equation*}
The last claim then follows by choosing 
\begin{equation*}
 m^-:= \lambda^{- N(m-2)/(N+m)}\frac{m \, \omega_m}{m+2} \frac{r^{2m-4}}{(2 \pi)^{m-2} } \left( (N+m) \log(5/4)\right)^{(m+2)/2}.
\end{equation*}

We eventually remove the assumption $\div \, b - c = 0$. We mainly rely on the steps in the display \eqref{e-core-h} and we point out the needed changes for this more difficult situation. With this aim, we recall that $| \div \, b(z) - c(z)| \le k := (N+1) \Lambda $ because of \eqref{e-up}. We next introduce two auxiliary functions
\begin{equation*}
 \widehat u(x,t) := e^{k(t-t_0)} u(x,t), \qquad \widetilde u(x,t) := e^{-k(t-t_0)} u(x,t).
\end{equation*}
Note that, as $\L u = 0$, we have that 
\begin{equation*}
 \widehat \L \, \widehat u : = \L \widehat u + k \widehat u = 0, \qquad 
 \widetilde \L \, \widetilde u : = \L \widetilde u - k \widetilde u = 0.
\end{equation*}
Note that $\widehat \L, \widetilde \L$ satisfy the condition \eqref{e-up} with $\Lambda$ replaced by $k$. In particular, the statemets \emph{i)}-\emph{iv)} hold for $\L, \widehat \L$ and $\widetilde \L$ with the same constants $r_0, \vartheta, M^+, m^-$. 

We denote by $\widehat M^{(m)}, \widetilde M^{(m)}$ the kernels relative to $\widehat \L, \widetilde \L$, respectively, and $\widehat \Omega^{(m)}, \widetilde \Omega^{(m)}$ the superlevel sets we use in the representation formulas appearing in \eqref{e-core-h}. As we did before,  we let $r^*$ be the constant appearing in Lemma \ref{lem-localestimate*} and we choose $r_0 := r^*/2$. As a direct consequence of the definiton of $\widehat u$ and  $\widetilde u$, there exist two positive constants $\widehat c$ and $\widetilde c$ such that
\begin{equation} \label{e-bounds-u}
 \widehat c u(z) \le \widehat u(z) \le u(z), \qquad u(z) \le \widetilde u(z) \le \widetilde c u(z),
\end{equation}
for every $z \in \Omega_{5r}^{(m)}(z_0)$. 

We are now in position to conclude the of proof Proposition \ref{prop-Harnack}. Let's consider the first two lines of \eqref{e-core-h}. Since $\widehat u$ is a solution to $\widehat \L \, \widehat u = 0$, for every $z \in K^{(m)}_{r}(z_0)$, it holds
\begin{equation*} 
 \widehat u(z) \le \frac{1}{(\vartheta r)^{N+m}} \int_{\widehat \Omega^{(m)}_{\vartheta r}(z)} 
 \! \! \! \! \! \! \widehat M_{\vartheta r}^{(m)} (z; \zeta) \widehat u(\zeta) \, d\zeta \le 
 \frac{M^+}{(\vartheta r)^{N+m}} 
 \int_{\widehat \Omega^{(m)}_{\vartheta r}(z)} \!\!\! \widehat u(\zeta) \, d\zeta.
\end{equation*}
The first inequality follows from the fact that $\div \, b(\zeta) - c(\zeta) - k \le 0$ for every $\zeta$. From \eqref{e-bounds-u} it then follows that
\begin{equation*} 
 u(z) \le \frac{M^+}{\widehat c \, (\vartheta r)^{N+m}} 
 \int_{\widehat \Omega^{(m)}_{\vartheta r}(z)} \!\!\! u(\zeta) \, d\zeta.
\end{equation*}
Continuing along the next lines of \eqref{e-core-h}, we note that {\it iii)} also holds in this form: $\widehat \Omega_{\vartheta r}^{(m)}(z) \subset \widetilde \Omega_{4r}^{(m)}(z_0) \cap \big\{\tau \le t_0 - \frac{r^2}{4 \pi \lambda^{N/(N+m)}}\big\}$ for every $z \in K^{(m)}_{r}(z_0)$, so that
\begin{equation*} 
 u(z) \le \frac{M^+}{\widehat c \, (\vartheta r)^{N+m}} 
 \int_{\widetilde \Omega_{4r}^{(m)}(z_0) \cap \big\{\tau \le t_0 - \frac{r^2}{4 \pi \lambda^{N/(N+m)}}\big\}} 
 \!\! u(\zeta) \, d\zeta.
\end{equation*}
On the other hand, using the fact that $\widetilde \L \, \widetilde u = 0$, and $\div \, b(\zeta) - c(\zeta) + k \ge 0$, we find
\begin{equation*}
 \frac{m^-}{(5 r)^{N+m}} \int_{\widetilde \Omega_{4r}^{(m)}(z_0) \cap \big\{\tau \le t_0 - \frac{r^2}{4 \pi \lambda^{N/(N+m)}}\big\}} 
 \!\!\! \widetilde u(\zeta) \, d\zeta \le
 \frac{1}{(5 r)^{N+m}} \int_{\widetilde \Omega_{5r}^{(m)}(z_0)} \! \! \! \! \! \! \widetilde M_{5 r}^{(m)} (z_0; \zeta) \widetilde u(\zeta) \, d\zeta \le 
  \widetilde u(z_0).
\end{equation*}
Thus, recalling that $u \le \widetilde u$ and $u (z_0) = \widetilde u(z_0)$, we conclude that
\begin{equation*}
 u(z) \le
 \frac{5^{N+m} M^+}{\widehat c \,  \vartheta^{N+m} m^-} u(z_0),
\end{equation*}
for every $z \in K^{(m)}_{r}(z_0)$. This concludes the proof of Proposition \ref{prop-Harnack}.
\end{proof}

\medskip

As a simple consequence of Proposition \ref{prop-Harnack} we obtain the following result. 

\begin{corollary} \label{cor-Harnack-inv}
There exist four positive constants $r_1, \kappa_1, \vartheta_1$ and $C_D$, with $\kappa_1, \vartheta_1 < 1$, such that the following inequality holds. For every $z_0 \in \Omega$ and for every positive $r$ such that $r \le r_1$ and ${\Q_{r}(z_0)} \subset \Omega$ we have that
  \begin{equation} \label{e-HD}
	\sup_{D_{r}(z_0)} u \le C_D u(z_0)
\end{equation}
for every $u \ge 0$ solution to $\L u = 0$ in $\Omega$. Here 
  \begin{equation*}
	D_{r}(z_0) := B_{\vartheta_1 r}(x_0) \times \{t_0 - \kappa_1 r^2\}.
\end{equation*}
\end{corollary}

\begin{center}
\begin{tikzpicture}
\clip(-.52,6.82) rectangle (7.02,1.48);
\path[draw,thick] (-.5,6.8) rectangle (7,1.5);
\begin{axis}[axis y line=none, axis x line=none, 
    xtick=\empty,ytick=\empty, 
    ymin=-1.1, ymax=1.1,
    xmin=-.2,xmax=1.8, samples=101, rotate= -90]
\addplot [black,line width=.7pt, domain=-.01:.01] {sqrt(.0001 - x * x)} node[above] {$z_0$};
\addplot [black,line width=.7pt, domain=-.01:.01] {-sqrt(.0001 - x * x)};
\addplot [ddblue,line width=.7pt, domain=.001:1] {sqrt(- 2 * x * ln(x))}; 
\addplot [ddblue,line width=.7pt,domain=.001:1] {- sqrt(- 2 * x * ln(x))} node[below] { \hskip30mm {\color{dddblue} $\Omega_r(z_0)$}};
\addplot [bblue,line width=.7pt,domain=.001:.1111] {sqrt(- 3 * x * ln(9*x)};
\addplot [bblue,line width=.7pt,domain=.001:.1111] {- sqrt(- 3 * x * ln(9*x))};
\end{axis}
\draw [line width=.6pt] (0,6.165) rectangle node [above=1.5cm,right=2.8cm] {$Q_r(z_0)$} (5.65,2);
\draw [red,line width=1.2pt] (2.1,5.9) -- node [below=5pt] { \hskip20mm $D_r(z_0)$} (3.6,5.9);
\end{tikzpicture} 

{\sc \qquad Fig.5}  - The set $D_r(z_0)$.
\end{center}

The above assertion follows from the fact that there exists a positive constant $\delta_1$ such that $\Omega_{r}^{(m)}(z_0) \subset \Q_{\delta_1 r}(z_0)$ for every $r \in 0], r_0[$ and that $D_{r}(z_0) \subset K^{(m)}_{r}(z_0)$, for some positive $\kappa_1, \vartheta_1$. Note that Corollary  \ref{cor-Harnack-inv} and Theorem \ref{th-Harnack-inv} differ in that, unlike the cylinders $\Q^+_{r}(z_0)$ and $\Q^-_{r}(z_0)$, the set $D_{r}(z_0)$ is not arbitrary. We next prove Theorem \ref{th-Harnack-inv} by using iteratively the Harnack inequality proved in Corollary \ref{cor-Harnack-inv}.

\medskip

\begin{proof} {\sc of Theorem \ref{th-Harnack-inv}.} As a first step we note that, up to the change of variable $v(x,t) := u(x_0 + r t, t_0 +r^2 t)$, it is not restrictive to assume that $z_0 = 0$ and $r = 1$. Indeed, the function $v$ is a solution to an equation $\widehat \L v = 0$, where the coefficients of the operator $\widehat \L$ are $\widehat a_{ij}(x,t) = a_{ij}(x_0 + r t, t_0 +r^2 t)$ satisfy all the assumptions made for $\L$, with the constants $M$ and $\Lambda$ appearing in \eqref{e-hc} and \eqref{e-up} replaced by $r^\alpha M$ and $r^\alpha \Lambda$, respectively, and the same constant $\lambda$ in \eqref{e-up}. Then, as $r \in ]0, R_0]$, the H\"older constant in \eqref{e-hc} of $\widehat \L$ is $R_0^{\, \alpha} M$ and the parabolicity constants in \eqref{e-up} are $\lambda$ and $R_0^{\, \alpha} \Lambda$, for every $r \in ]0, R_0]$. In the following we then assume that $z_0 = 0$ and $r=1$. Moreover, $r_1$ denotes the constant appearing in Corollary \ref{cor-Harnack-inv} and relative to $\widehat \L$, which depends on the constants $M, \lambda, \Lambda$ and $R_0$. We then choose four positive constants $\iota, \kappa, \mu, \vartheta$ with $0 < \iota < \kappa < \mu < 1$ and $0 < \vartheta < 1$ and we consider the cylinders $\Q^+ := \Q_1^+(0)$ and $\Q^- := \Q_1^-(0)$ as defined in \eqref{e-QPM}. We  let 
\begin{equation*}
	r_0 := \min \big\{r_1, 1 - \vartheta, \sqrt{1 - \mu} \big\}
\end{equation*}
and we note that $\Q_r(z) \subset \Q_1(0)$ whenever $z  \in B(0,\vartheta) \times ]- \mu, 0[$ and $0 < r < r_0$.

We next choose any $z^- = (x^-,t^-) \in \Q^-, z^+ = (x^+,t^+) \in \Q^+$ and we rely on Corollary \ref{cor-Harnack-inv} to construct a \emph{Harnack chain}, that is a finite sequence $w_0, w_1, \dots, w_m$ in $\Q_1(0)$ such that
\begin{equation} \label{eq-HC}
	w_0 = z^+, \qquad w_k = z^-, \qquad u(w_j) \le C_D u(w_{j-1}), \quad j=1, \dots, m.
\end{equation}

\begin{center}
\begin{tikzpicture}
\path[draw,thick] (-1,.7) rectangle (8.7,6.5);
\filldraw [fill=black!4!white, line width=.6pt] (1.5,6) rectangle node[right=2cm] {$Q^+_r(z_0)$} (5.5,4.5);
\filldraw [fill=black!4!white, line width=.6pt] (1.5,2) rectangle node[right=2cm] {$Q^-_r(z_0)$} (5.5,3.5);
\draw [line width=.6pt] (0,6) rectangle node[right=3.5cm] {$Q_r(z_0)$}(7,1);
\draw [line width=.6pt] (3.5,6) circle (1pt) node[above] {$z_0$};
\foreach \x in {0,1,...,5}
\draw [xshift=\x*.4 cm,yshift=-\x*.2 cm,bblue,line width=.6pt] (.4,3.6) rectangle (3.2,4.6);
\foreach \x in {0,1,...,5}
\draw [xshift=\x*.4 cm,yshift=-\x*.2 cm,red,line width=1.2pt] (1.4,4.4) -- (2.2,4.4);
\foreach \x in {0,1,...,5}
\draw [xshift=\x*.4 cm,yshift=-\x*.2 cm,line width=1pt] (2.2,4.4) circle (1pt);
\draw [line width=1pt] (1.8,4.6) circle (1pt) node[above] {$z^+$};
\draw [line width=1pt] (4.2,3.4) circle (1pt) node[above=2pt] {$\, z^-$};
\end{tikzpicture}

{\sc \qquad Fig.6}  - A Harnack chain.
\end{center}

We build a Harnack chain as follows. For a positive integer $m$ that will be fixed in the sequel, we choose a positive $r$ and the vector $y \in \R^N$ satisfying
\begin{equation} \label{eq-y}
	m \kappa_1 r^2 = t^+-t^-, \qquad m r y = x^+ - x^-.
\end{equation}
Let $\kappa_1,\ \vartheta_1$ be the constants in Corollary \ref{cor-Harnack-inv}. We define
\begin{equation} \label{eq-wj}
	w_j := (x^+ + j r y, t^+ - j \kappa_1 r^{2}), \qquad j=0,1, \dots, m.
\end{equation}
Clearly, if $r \le r_0$, then $\Q_r(w_j) \subset \Q_1(0)$ for $j=0, 1, \dots, m$. If moreover $|y| \le \vartheta_1$, then $w_{j} \in D_{r}(w_j-1)$ for $j=1, \dots, m$. This proves that \eqref{eq-HC} holds, and we conclude that
\begin{equation} \label{eq-H+-}
	u(z^-) \le C_D^{\, m} u(z^+).
\end{equation}
We next choose $m$ in order to have both condtions $r \le r_0$ and $|y| \le \vartheta_1$ satisfied. 

The choice of $m$ is different in the case $|x^+ - x^-|$ is \emph{small} or \emph{large} with respect to $t^+-t^-$. If 
\begin{equation} \label{eq-case1}
	\frac{|x^+-x^-|}{t^+-t^-} \le \frac{\vartheta_1}{\kappa_1 r_0},
\end{equation}
we let $m$ be the positive integer satisfying
\begin{equation} \label{eq-mm}
	(m-1) \kappa_1 r_0^{\, 2} < t^+ - t^- \le m \kappa_1 r_0^{\, 2},
\end{equation}
and, in accordance with \eqref{eq-y}, we choose $r$ as the unique positive number satisfying $m \kappa_1 r^2 = t^+-t^-$. From \eqref{eq-mm} it directly follows $r \le r_0$, while from \eqref{eq-mm} and \eqref{eq-case1} we obtain $|y| \le \vartheta_1$.

Suppose now that
\begin{equation} \label{eq-case2}
	\frac{|x^+-x^-|}{t^+-t^-} > \frac{\vartheta_1}{\kappa_1 r_0}.
\end{equation}
In view of \eqref{eq-y}, in this case we choose $m$ as the integer satisfying
\begin{equation} \label{eq-m}
	m - 1 < \frac{\kappa_1 |x^+-x^-|^2}{\vartheta_1^{\, 2} (t^+-t^-)} \le m,
\end{equation}
and we let $y$ be the vector parallel to $x^- - x^+$ and such that 
\begin{equation*} 
	m |y| = \frac{\kappa_1 |x^+-x^-|^2}{{\vartheta_1 (t^+-t^-)}}.
\end{equation*}
Clearly, $|y| \le \vartheta_1$, and \eqref{eq-case2} implies $r \le r_0$. 

We next find a bound for the integer $m$, which is uniform with respect to $z^- \in \Q^-$ and $z^+ \in \Q^+$, and we rely on \eqref{eq-H+-} to conclude the proof. In the first case \eqref{eq-case1} we obtain from \eqref{eq-mm} that $m \le \frac{t^+-t^-}{\kappa_1 r_0^{\, 2}}$. In the second case \eqref{eq-case2} we rely on \eqref{eq-m} and we note that $t^+-t^- \ge \kappa- \iota$, by our choiche of $\Q^-$ and $\Q^+$. Then in this case we have $m < \frac{ 4\kappa_1}{\vartheta_1^{\, 2} (\kappa - \iota)}$
Summarizing, we have proved that the inequality \eqref{e-H1} holds with 
\begin{equation*}
 C_H := \exp \left( \max \big\{ \tfrac{1}{\kappa_1 r_0^{\, 2}}, 
 \tfrac{4\kappa_1}{\vartheta_1^{\, 2} (\kappa - \iota)} \big\} \log C_D \right).
\end{equation*}
\end{proof}

\setcounter{equation}{0} 
\section{An approach relying on sets of finite perimeter}\label{SectionBV}

In this section we present another approach to the generalized divergence theorem, relying on 
De Giorgi's theory of perimeters, see \cite{DeGiorgi2,DeGiorgi3} or 
\cite{AmbrosioFuscoPallara,Maggi},
and we show how this leads to a slightly different proof of Theorem \ref{th-1}. 
This approach requires more prerequisites than that used in Section \ref{SectionDivergence}, but, 
as explained in the Introduction, is more flexible and avoids the Dubovicki\v{\i} theorem. 
In this section, if $\mu$ is a Borel measure and $E$ is a Borel set, we use the notation 
$\mu\mres E(B)=\mu(E\cap B)$. 
As before, $C^1_c \pr{\Omega}$ denotes the set of $C^1$ functions compactly supported in the open set $\Omega \subset \R^n$. 

\begin{definition}[\sc $BV$ Functions]
Let $u \in L^1 \pr{\Omega}$; we say that $u$ is a function of bounded variation in $\Omega$ if its
distributional derivative $Du=\pr{D_1 u,\ldots,D_n u}$ is an $\R^n$-valued Radon measure in $\Omega$, 
{\em i.e.}, if
\[
    \int_{\Omega}u \frac{\partial \varphi}{\partial z_i}\ dz=-\int_{\Omega}\varphi \, dD_i u, 
    \quad \forall \: \varphi \in C_c^1 \pr{\Omega}, \; i=1,\ldots,n 
\]
or, in vectorial form, 
\begin{equation} \label{e-bv-div}
 \int_{\Omega}u\, \mathrm{div}\,\Phi\ dz=-\sum_{i=1}^n \int_{\Omega}\Phi_i\ d D_i u
 = -\int_{\Omega}\langle\Phi, Du\rangle , 
\quad \forall \: \Phi \in C_c^1 \pr{\Omega;\R^n}.
\end{equation}
The vector space of all functions of bounded variation in $\Omega$ is denoted by $BV \pr{\Omega}$. 
The variation $V \pr{u,\Omega}$ of $u$ in $\Omega$ is defined by:
$$
    V \pr{u,\Omega}:=\sup \left\{  \int_{\Omega}u\, \mathrm{div}\,\Phi\ dz:
    \Phi \in C^1_c \pr{\Omega;\R^{n}},\, \norm{\Phi}_{\infty} \leq 1 \right\}.
$$
\end{definition}
We recall that ${V \pr{u,\Omega}=\abs{Du}\pr{\Omega}}<\infty$ for any $u \in BV \pr{\Omega}$, where 
$\abs{Du}$ denotes the total variation of the measure $Du$. We also recall that if 
$u \in C^1 \pr{\Omega}$ then 
$$
V \pr{u,\Omega}=\int_{\Omega} \abs{\nabla u} dz.
$$
When the function $u$ is the characteristic functions $\chi_E$ of some measurable set, its variation is said \emph{ perimeter of $E$}.
\begin{definition}[\sc Sets of finite perimeter]
Let $E$ be an $\mathcal{L}^n-$measurable subset of $\R^n$. For any open set $\Omega \subset \R^n$ the perimeter of $E$ in $\Omega$ is denoted by $P \pr{E,\Omega}$ and it is the variation of $\rchi_{E}$ in 
$\Omega$, {\em i.e.},
\begin{equation*} 
P \pr{E,\Omega}:=\sup \left\{\int_{E}\mathrm{div}\, \Phi\ dz:\Phi \in C^1_c \pr{\Omega;\R^n},\, \norm{\Phi}_{\infty}\leq 1 \right\}.
\end{equation*}
We say that $E$ is a set of finite perimeter in $\Omega$ if $P \pr{E,\Omega}<\infty$. 
\end{definition}
Obviously, several properties of the perimeter of $E$ can be stated in terms of the variation of $\chi_E$. In particular, if $\mathcal{L}^n(E \cap \Omega)$ is finite, then $\rchi_{E} \in L^1 \pr{\Omega}$ and $E$ has finite perimeter in $\Omega$ if and only if $\rchi_{E}\in BV \pr{\Omega}$ and $P \pr{E,\Omega} = \abs{D{\rchi_{E}}}\pr{\Omega}$. Both the notations $|D\chi_E|(B)$ and $P(E,B)$, $B$ Borel, are used 
to denote the total variation measure of $\chi_E$ on a Borel set $B$ and we say that $E$ is a set of locally
finite perimeter in $\Omega$ if $P(E,K)<\infty$ for every compact set $K \subset \Omega$. Finally, formula \eqref{e-bv-div} looks like a divergence theorem:
\begin{equation} \label{e-per-div}
 \int_{E}\, \mathrm{div}\, \Phi\ dz=- \int_{\Omega}\langle\Phi,D\chi_E\rangle, 
\quad \forall \: \Phi \in C_c^1 \pr{\Omega;\R^n}, 
\end{equation}
but it becomes more readable if some precise information is given on the set where the measure $D\chi_E$ 
is concentrated. Therefore, we introduce the notions of {\em reduced boundary} and of {\em density} and  
recall the structure theorem for sets with finite perimeter due to E. De Giorgi, see \cite{DeGiorgi3} 
and \cite[Theorem 3.59]{AmbrosioFuscoPallara}, and the characterization due to H. Federer.  

\begin{definition}[\sc Reduced boundary]
Let $\Omega$ be an open subset of $\R^n$ and let $E$ be a set of locally finite perimeter in $\Omega$. 
We say that $z \in \Omega$ belongs to the reduced boundary ${\mathcal F}E$ of $E$ if $|D\chi_E|(B_\r(z))>0$ 
for every $\r>0$ and the limit
$$
    \nu_E\pr{z}:=\lim_{\r \to 0^+}\frac{D{\rchi_{E}}\pr{B_{\r}\pr{z}}}{\abs{D{\rchi_{E}}}\pr{B_{\r}\pr{z}}}
$$
exists in $\R^n$ and satisfies $\abs{\nu_E \pr{z}}=1$. The function 
$\nu_E:{\mathcal F}E \rightarrow {\mathbb S}^{n-1}$ 
is Borel continuous and it is called the {\em generalized (or measure-theoretic) inner normal} to $E$. 
\end{definition}
Notice that the reduced boundary is a subset of the topological boundary. The Besicovitch differentiation 
theorem, see e.g. \cite[Theorem 2.22]{AmbrosioFuscoPallara}, yields  $D\rchi_{E} = \nu_E\abs{D\rchi_{E}}$, 
and $\abs{D\rchi_{E}}(\Omega\setminus{\mathcal F}E)=0$, hence \eqref{e-per-div} becomes
\begin{equation} \label{e-3} 
\int_{E}\mathrm{div}\, \Phi\ dz=-\int_{{\mathcal F}E}\scp{\nu_E,\Phi}\ d \abs{D\rchi_{E}}, 
\quad \forall \: \Phi \in C^1_c \pr{\Omega;\R^N}.
\end{equation}

The relation between the topological boundary and the reduced boundary can be further analyzed by 
introducing the notion of {\em density} of a set at a given point. 

\begin{definition}[\sc Points of density $\alpha$] 
For every $\alpha \in \left[0,1\right]$ and every $\mathcal{L}^n-$measurable set $E \subset \R^n$ we denote by $E^{\pr{\alpha}}$ the set
$$
E^{\pr{\alpha}}=\left\{ z \in \R^n: \lim_{\r \to 0^+}
\frac{\mathcal{L}^n(E \cap B_\r \pr{z})}{\mathcal{L}^n(B_\r \pr{z})}=\alpha\right\}.
$$
\end{definition}
Thus $E^{\pr{\alpha}}$, which turns out to be a Borel set, is the set of all points where $E$ has density 
$\alpha$. The sets $E^{\pr{0}}$ and $E^{\pr{1}}$ are called \emph{the measure-theoretic exterior} and \emph{interior} of $E$ and, in general, strictly contain the topological exterior and interior of the set $E$, respectively. 
We recall the well known \emph{Lebesgue's density theorem}, that asserts that for every 
$\mathcal{L}^n-$measurable set $E \subset \R^n$
$$
\mathcal{L}^n(E \triangle E^{\pr{1}})=0, \quad \mathcal{L}^n(\pr{\R^n \setminus E} \triangle E^{\pr{0}})=0,
$$
{\em i.e.}, the density of $E$ is $0$ or $1$ at $\mathcal{L}^n-$almost every point in $\R^n$. 
This notion allows to introduce the 
{\em essential} or {\em measure-theoretic} boundary of $E$ as 
$\partial^*E=\R^n\setminus (E^{\pr{0}}\cup E^{\pr{1}})$, which is contained in the topological boundary
and contains the reduced boundary. Finally, the De Giorgi structure theorem says that 
$|D\rchi_E|=\H^{n-1}\mres{\mathcal F}E$ and a deep result due to Federer 
(see \cite[4.5.6]{federer1969geometric} or 
\cite[Theorem 3.61]{AmbrosioFuscoPallara}) states that if $E$ has finite perimeter in $\R^n$ then 
\[
\mathcal{F}E\subset E^{\pr{1/2}}\subset \partial^*E \quad \text{and}\quad 
\H^{n-1}(\R^n\setminus (E^{\pr{0}}\cup\mathcal{F}E\cup E^{\pr{1}}))=0
\]
hence, in particular, $\nu_E$ is defined $\H^{n-1}-$a.e. in $\partial^*E$. Notice also (see 
\cite[Theorem 3.62]{AmbrosioFuscoPallara}) that if $\H^{n-1}(\partial E)<\infty$ then $E$ has finite 
perimeter. The results of De Giorgi and Federer imply that if $E$ is a set of finite perimeter in 
$\Omega$ then $D \rchi_E=\nu_E \H^{n-1}\mres \mathcal{F}E$ and the divergence 
theorem \eqref{e-3} can be rewritten in the form:
\begin{equation}\label{e-rem-1}
\int_E \mathrm{div}\, \Phi\ dz=
-\int_{{\mathcal F}E}\scp{\nu_E,\Phi}\ d \H^{n-1}=
-\int_{\partial^*E}\scp{\nu_E,\Phi}\ d \H^{n-1}, \quad \forall \: \Phi \in C^1_c \pr{\Omega;\R^n},
\end{equation}
much closer to the classical formula \eqref{eq-div}. Indeed, the only difference is that the inner normal and the boundary are understood in a measure-theoretic sense and not in the topological one; in particular, for a generic set of finite perimeter, ${\mathcal F}E$ needs not to be closed and $\nu_E$ needs not to be continuous.
Moreover, $\nu_E$ is defined $\H^{n-1}-$a.e. in $\partial^*E$. 

Let us see now how we can rephrase the results of Section \ref{SectionDivergence} in terms of
perimeters and how we can modify the proof of Theorem \ref{th-1}. 
We first recall the Fleming--Rischel formula 
(see \cite{fleming} or \cite[Theorem 3.40]{AmbrosioFuscoPallara}), {\em i.e.}, the coarea
formula for $BV$ functions.
\begin{theorem}[\sc Coarea formula in $BV$] \label{th-cobv}
For any open set $\Omega \subset \R^n$ and $G \in L^1_{\mathrm{loc}}\pr{\Omega}$ one has
$$
V \pr{G,\Omega}=\int_{\R}P\pr{\{z \in \Omega:G \pr{z}>y\},\Omega}\ dy.
$$
In particular, if $G \in BV \pr{\Omega}$ the set $\left\{G >y \right\}$ has finite perimeter in $\Omega$ for 
$\H^1-$a.e.~$y \in \R$ and
$$
\abs{DG}\pr{B}=\int_{\R}\abs{D{\rchi_{\left\{G>y\right\}}}}\pr{B}\ dy, \quad 
DG\pr{B}=\int_{\R}D{\rchi_{\left\{G>y\right\}}}\pr{B}\ dy, \quad 
\forall \: B \in \mathcal{B}\pr{\Omega}.
$$
\end{theorem}

Now we are ready to state the analogue of Proposition \ref{prop-1} and to prove Theorem \ref{th-1}
again.

\begin{proposition} \label{prop-1BV}
Let $\Omega$ be an open subset of $\R^{n}$ and let $F \in BV \pr{{\Omega};\R}\cap C \pr{{\Omega};\R}$.
Then, for $\H^1-$almost every $y \in \R$, we have:
\begin{equation}\label{cobv}
    \int_{\left\{F>y\right\}} \mathrm{div}\, \Phi\ dz =
    -\int_{\partial^*\{F>y\}} \scp{\nu,\Phi}\ d \H^{n-1},
    \quad \forall \: \Phi \in C_c^1 \pr{\Omega;\R^{n}},
\end{equation}
were $\nu$ is the generalized inner normal to $\{F>y\}$.
\end{proposition}

\begin{proof} By Theorem \ref{th-cobv}, 
the set $\left\{F>y\right\}$ has finite perimeter in $\Omega$ for $\H^1-$a.e. $y \in \R$, 
hence we may apply \eqref{e-rem-1} with $E=\{F>y\}$ and conclude.
\end{proof}

\medskip

As in Section \ref{SectionDivergence}, we have to cut the integration 
domain: therefore, we study the intersection between the super-level set of a 
generic function $G \in BV\pr{\Omega;\R}\cap C\pr{\Omega;\R}$ and a half-space 
$H_t=\left\{ x \in \R^n: \scp{x,e}<t \right\}$, 
for some $e \in {\mathbb S}^{n-1}$, $t \in \R$. First, we present a general formula that characterizes the 
intersection of two sets of finite perimeter for which we refer to Maggi's book, see \cite[Theorem 16.3]
{Maggi}. 

\begin{theorem}[\sc Intersection of sets of finite perimeter] \label{th-int}
If $A$ and $B$ are sets of locally finite perimeter in $\Omega$, and we let
$$
\left\{ \nu_A = \nu_B \right \} = 
\left\{ x \in {\mathcal F}A \cap {\mathcal F}B: \nu_A \pr{x}=\nu_B \pr{x} \right\},
$$
then $A \cap B$ is a set of locally finite perimeter in $\Omega$, with
\begin{equation} \label{e-9}
D \rchi_{A \cap B}=D \rchi_A \mres B^{\pr{1}}+D \rchi_B \mres A^{\pr{1}}
+ \nu_A \H^{n-1} \mres \left\{ \nu_A = \nu_B \right\}. 
\end{equation}
\end{theorem}
In the case in which $B$ is a half-space, formula \eqref{e-9} can be greatly simplified; 
indeed we can prove the following corollary.  

\begin{corollary}[\sc Intersections with a half-space]\label{cor-1}
Let $E$ be a set of locally finite perimeter in $\Omega$ and let 
$H_t=\left\{ z \in \R^n : \scp{z,e} < t\right\}$ for some $e \in {\mathbb S}^{n-1}$, $t \in \R$. Then, for every 
$t \in \R$, $E \cap H_t$ is a set of locally finite perimeter in $\Omega$ and moreover, 
for $\H^{1}-$almost every $t \in \R$, 
$$
D \rchi_{E \cap H_t}=D \rchi_E \mres H_t-e\H^{n-1}\mres \pr{E \cap \left\{ \scp{x,e}=t\right\}}.
$$
\end{corollary}
\begin{proof}
The half-space $H_t$ is clearly a set of locally finite perimeter in $\Omega$ for every $t\in\R$, 
and for every $t \in \R$ we have, $H_t^{\pr{1}}=H_t$, 
$\mathcal{F}H_t=\partial H_t =\left\{ \scp{x,e}=t \right\}$ and $\nu_{H_t} \equiv -e$. 
Then, applying Theorem \ref{th-int} we see that $E \cap H_t$ is a set of locally finite perimeter 
in $\Omega$ for every $t \in \R$ and \eqref{e-9} reads
\[
D \rchi_{E \cap H_t}=D \rchi_E \mres H_t - e \H^{n-1} \mres (E^{\pr{1}}\cup 
\left\{ \nu_E = \nu_{H_t} \right\}) .
\]
Since by Fubini theorem 
\[
0=\mathcal{L}^n (E\triangle E^{(1)}) = 
\int_{\R} \H^{n-1}\pr{(E\triangle E^{(1)}) \cap \left\{ \scp{x,e}=t \right\}}\ dt ,
\]
for $\H^1-$a.e. $t \in \R$ we have 
\[
\H^{n-1}\pr{E \triangle E^{\pr{1}} \cap \left\{ \scp{x,e}=t \right\}}=0.
\]
Therefore, 
\[
D \rchi_{H_t} \mres E = D \rchi_{H_t} \mres E^{\pr{1}} = -e\H^{n-1} \mres (E\cap\{\scp{x,e}=t\})
= -e\H^{n-1} \mres (E\cap\{\nu_E=\nu_{H_t}\})
\] 
for $\H^1-$a.e. $t \in \R$ and the thesis follows.
\end{proof}

\medskip

The following corollary allows us to perform (with some modifications) the last part of the 
proof of our main result.

\begin{corollary}\label{prop-2-BV}
Let $\Omega=\R^{N+1} \setminus \left\{ \pr{z_0} \right\}$, $G \in BV\pr{\Omega;\R} 
\cap C \pr{\Omega;\R}$, 
$H_t=\left\{ z \in \R^{N+1} : \scp{z,e} < t\right\}$ for some $e \in {\mathbb S}^N$, $t \in \R$. Then, 
for $\H^{1}-$almost every $w \in \R$ and for every $t < t_0$ the set 
$E \cap H_t$ has locally finite perimeter in $\Omega$ and 
$$
\int_{\left\{ G > w \right\} \cap H_t}\mathrm{div}\, \Phi\ dz=
-\int_{\partial^*\{ G > w \} \cap H_t}\scp{\nu,\Phi}\ d \H^N
+\int_{\left\{ G >w \right\} \cap \left\{ \scp{x,e}=t \right\}} \scp{e,\Phi}\ d\H^N,
$$
for every $\Phi \in C^1_c \pr{\Omega;\R^n}$, where $\nu$ is the generalized inner normal to 
$\partial^*(\{G>w\})$.
In particular, if $e=\pr{0,\ldots,0,1}$, for every $\varepsilon > 0$
$$
\int_{\left\{ G > w \right\} \cap \left\{ t<t_0-\varepsilon \right\}}
\!\!\!\mathrm{div}\, \Phi\ dz
=-\int_{\partial^*\{ G>w \} \cap \left\{ t<t_0-\varepsilon \right\}}
\!\!\scp{\nu,\Phi}d \H^{N}
+\int_{\left\{ G>w \right\} \cap \left\{ t=t_0-\varepsilon \right\}} 
\!\!\scp{e,\Phi}d\H^{N}. 
$$
\end{corollary}

Notice that the difference between Proposition \ref{prop-1bis} and 
Corollary \ref{prop-2-BV} is that in the former we can exclude the set of critical points of $G$ from 
the surface integral, thanks to Dubovicki\v{\i} theorem, and we know that $\nu$ is given by the 
normalized gradient of $G$ {\em everywhere} in the integration set, whereas in the latter we don't 
need to know any estimate on the size of ${\rm Crit}(G)$ and $\nu$ is defined $\H^N-${\em a.e.} on 
the integration set (still coinciding with the normalized gradient of $G$ out of ${\rm Crit}(G)$, 
of course). First, notice that we apply Corollary \ref{prop-2-BV} to $G(z)=\Gamma(z_0;z)$, which is 
$C^1(\Omega)$, hence Lipschitz on bounded sets. As a consequence, $\partial\{G>w\}\subseteq\{G=w\}$ 
and comparing the coarea formulas \eqref{e-co} and \eqref{cobv}, we deduce that 
$\H^N(\{G=w\}\setminus\partial^*\{G>w\})=0$ for $\H^{1}$-a.e. $w$. 
Let us see how this entails modifications of the 
proof of Theorem \ref{th-1}: the proof goes in the same vein until \eqref{eq-div-3i}, 
\eqref{eq-div-3}, which in the present context are replaced by
\begin{align*}
 \lim_{k \to +\infty} \int_{\psi_r(z_0) \cap \left\{ t<t_0-\varepsilon_k \right\}}
\scp{\nu,\Phi}d \H^{N} & = \int_{\psi_r(z_0)} K (z_0;z) u(z)  d \H^{N}
\\
& = \int_{\psi_r(z_0)\setminus \mathrm{Crit}\pr{\Gamma}} K (z_0;z) u(z)  d \H^{N}
\end{align*} 
where the first equality follows from Corollary \ref{prop-2-BV}, as explained, and the last equality
follows from the fact that the kernel $K$ vanishes in ${\rm Crit}(\Gamma)$. 
The rest of the proof needs no modifications.

\def\cprime{$'$} \def\cprime{$'$} \def\cprime{$'$}
  \def\lfhook#1{\setbox0=\hbox{#1}{\ooalign{\hidewidth
  \lower1.5ex\hbox{'}\hidewidth\crcr\unhbox0}}} \def\cprime{$'$}
  \def\cprime{$'$}

\end{document}